\documentclass{article}
\usepackage[utf8]{inputenc}
\usepackage[english]{babel}
\usepackage{amsmath}
\usepackage{amsfonts}
\usepackage{amssymb}
\usepackage{mathtools}
\usepackage{amsthm}
\usepackage{float}
\usepackage{etoolbox}
\usepackage{bbm}
\usepackage{hyperref}
\usepackage{graphicx}
\usepackage{xcolor}
\usepackage[T1]{fontenc}
\newtheorem{theorem}{Theorem}
\newtheorem{lemma}{Lemma}
\newtheorem{remark}{Remark}

\newtheorem{corollary}{Corollary}
\newtheorem{proposition}{Proposition}

\newtheorem{assumption}{Assumption}

\usepackage[left=2cm,right=2cm,top=2cm,bottom=2cm]{geometry}
\DeclareMathOperator*{\argmax}{arg\,max}

\newcommand{\probability}{\operatorname{\mathbb{P}}\probarg}
\DeclarePairedDelimiterX{\probarg}[1]{(}{)}{%
  \ifnum\currentgrouptype=16 \else\begingroup\fi
  \activatebar#1
  \ifnum\currentgrouptype=16 \else\endgroup\fi
}
\newcommand{\activatebar}{%
  \begingroup\lccode`\~=`\|
  \lowercase{\endgroup\let~}\innermid 
  \mathcode`|=\string"8000
}
\newcommand{\LimitN}{\overset{N\to\infty}{\longrightarrow}}
\newcommand{\LimitD}{\overset{d}{\longrightarrow}}
\newcommand{\LimitP}{\overset{\mathbb{P}}{\longrightarrow}}

\let\oldfrac\frac
\renewcommand{\frac}[2]{%
  \mathchoice
    {\oldfrac{#1}{#2}}
    {#1/#2}
    {#1/#2}
    {#1/#2}
}

\title{Maximum waiting time in heavy-tailed fork-join queues}
\author{Dennis Schol, Maria Vlasiou, and Bert Zwart}
\date{November 2022}

\begin{document}

\maketitle
\begin{abstract}
    In this paper, we study the maximum waiting time $\max_{i\leq N}W_i(\cdot)$ in an $N$-server fork-join queue with heavy-tailed services as $N\to\infty$. The service times are the product of two random variables. One random variable has a regularly varying tail probability and is the same among all $N$ servers, and one random variable is Weibull distributed and is independent and identically distributed among all servers. This setup has the physical interpretation that if a job has a large size, then all the subtasks have large sizes, with some variability described by the Weibull-distributed part. We prove that after a temporal and spatial scaling, the maximum waiting time process converges in $D[0,T]$ to the supremum of an extremal process with negative drift. The temporal and spatial scaling are of order $\tilde{L}(b_N)b_N^{\frac{\beta}{(\beta-1)}}$, where $\beta$ is the shape parameter in the regularly varying distribution, $\tilde{L}(x)$ is a slowly varying function, and $(b_N,N\geq 1)$ is a sequence for which holds that $\max_{i\leq N}A_i/b_N\LimitP 1$, as $N\to\infty$, where $A_i$ are i.i.d.\ Weibull-distributed random variables. Finally, we prove steady-state convergence.
\end{abstract}

\section{Introduction}\label{sec: intro}


The fork-join queue is a useful tool to model streams of jobs, consisting of subtasks, in parallel systems. A key quantity of interest is the behavior of the longest queue. In this paper, we investigate a fork-join queue with $N$ servers, where each of these servers has to complete a subtask of an incoming task. We assume that $N$ is large, and we investigate the longest waiting time among the $N$ subtasks. Moreover, we assume that service times are mutually dependent, and can be written as a product of two random variables, where one term is independent and identically distributed for all servers, and has a Weibull-like tail, while the other term is the same for all servers and has a regularly varying tail. This describes the situation that if a job has a large size, all the subtasks also have a large size, where the fluctuation is described by the Weibull-like distributed random variable.

We obtain a convergence result for the rescaled transient maximum waiting time $\max_{i\leq N}W_i(tc_N)/c_N$ as $N\to\infty$, after choosing the proper temporal and spatial scaling $(c_N,N\geq 1)$. This maximum waiting time converges in distribution to a process which is the supremum of Fr\'echet-distributed random variables minus a drift term. The temporal and spatial scaling $c_N$ depends on the extreme-value scaling of $N$ independent Weibull-distributed random variables, a slowly varying function, and the parameter of regular variation. Hence, to obtain this result, a mixture of classic extreme value theory and analysis of heavy tails is needed. Furthermore, we show that this rescaled maximum waiting time process $(\max_{i\leq N}W_i(tc_N)/c_N,t\in[0,T])$ converges as a process in $D[0,T]$ to an extremal process $(\sup_{s\in[0,t]}(X_{(s,t)}-\mu(t-s)),t\in[0,T])$ with Fr\'echet marginals, with $D[0,T]$ the space of cadlag functions on $[0,T]$, which we equip with the $d^0$ metric, cf.\ \cite[Eq.\ (12.16)]{billingsley2013convergence}. Finally, we prove steady-state convergence of $\max_{i\leq N}W_i(\infty)/c_N$ to $\lim_{t\to\infty}\sup_{s\in[0,t]}(X_{(s,t)}-\mu(t-s))$. 

Applications of these heavy-tailed fork-join queues are usually found in parallel computing. Companies such as Google, Microsoft and Alibaba have datacenters with thousands of servers that are available for cloud computing, where there is often a form of parallel scheduling. Jobs in these systems have typically large sizes, and are often heavy tailed. However, most literature on parallel queueing theory assumes service times to be light tailed, cf. the survey \ \cite{harchol2021open}. This motivates the analysis of parallel queueing networks with heavy-tailed job sizes.

Our work relates to the literature on fork-join queues. Exact results on probability distributions of fork-join queues are only derived for fork-join queues with two service stations, cf.\ \cite{baccelli1985two,flatto1984two,de1988fredholm,wright1992two}. Approximations and bounds for performance measure of the fork-join queue with an arbitrary but fixed number of servers can be found in \cite{baccelli1989queueing,ko2004response,nelson1988approximate}. In \cite{varma1990heavy}, a heavy-traffic analysis for fork-join queues is derived; see also \cite{nguyen1993processing} and \cite{nguyen1994trouble}. More recent work in this direction may be found in \cite{lu2015gaussian,lu2017heavy,lu2017heavy2,MeijerSchol2021,scholMOR}. 

Moreover, our work is connected to literature on heavy-tailed phenomena; cf.\ \cite{nair2022fundamentals} for a summary. Specific results on the interplay between fork-join queues and heavy-tailed services can be found in \cite{raaijmakers2021fork, xia2007scalability, zeng2021fork}. In \cite[Thm.\ 2]{raaijmakers2021fork}, asymptotic lower and upper bounds for the tail probability of the maximum waiting time in steady state are given; however these bounds are not sharp when $N$ is large. In \cite{xia2007scalability} and \cite{zeng2021fork}, the authors investigate the fork-join queue with heavy-tailed services under a blocking mechanism. This paper contributes to the existing literature, as we give sharp convergence results for the maximum waiting time with heavy-tailed service times, where the number of servers $N$ grows large. We combine results from extreme value theory, the analysis of heavy-tailed random variables, and results on process convergence in $D[0,T]$.

Furthermore, the limiting process in this paper has an interesting form; this process is an extremal process with negative drift. Several papers have been written on extremal processes, cf.\ \cite{ballerini1985records, ballerini1987records, dwass1964extremal, dwass1966extremal, resnick1973structure}. These results are, among others, used and applied on the analysis of records in sport, cf.\ \cite{ballerini1985records, ballerini1987records}. For example, in \cite{ballerini1985records}, a model is used to analyze the times in the mile run.  

This paper is organized as follows. We present our model in Section \ref{sec: model} and our main results in Theorems \ref{thm: auxiliary result waiting time heavy tail}, \ref{thm: main result waiting time heavy tail}, \ref{thm: steady state heavy tail}, and Proposition \ref{prop: marginal steady state}. We give a heuristic analysis of our results in Section \ref{subsec: heuristic analysis}. In Section \ref{subsec: other choices}, we discuss other modeling choices. In Section \ref{sec: prel results heavy tails}, we present some auxiliary results. We prove process convergence in Section \ref{sec: process convergence}. Finally, we prove our main results in Section \ref{sec: steady state}.

\section{Model and main results}\label{sec: model}
In this paper, we analyze a fork-join queue with a common arrival process, and a service process that consists of a Weibull-like i.i.d.\ part and a regularly varying part that is the same among all the servers. This models the situation that if a job has a large size, then all the subtasks have a large size, with some variability. We show in Section \ref{subsec: heuristic analysis} that the Weibull distribution has convenient properties that we exploit in this paper; in Section \ref{subsec: other choices}, we briefly discuss what happens when the i.i.d.\ part of the service process has a lighter tail. We write the random variable $A_{i,j}B_j$ as the representation of a service time at server $i$ of the subtask of the $j$-th job, while the random variable $T_j$ is the interarrival time between the $j$-th and $(j+1)$-st job. Now, by Lindley's recursion, the waiting time at server $i$ upon arrival of the $(n+1)$-st job equals
\begin{align}\label{eq: indiv waiting time}
W_i(n)=\sup_{0\leq k\leq n}\sum_{j=k+1}^n (A_{i,j}B_j-T_j),
\end{align}
with $W_i(0)=0$ and $\sum_{j=n+1}^n(A_{i,j}B_j-T_j)=0$. Moreover, we write $W_i(t)=W_i(\lfloor t\rfloor)$. Furthermore, the maximum of the $N$ waiting times equals
\begin{align}
\max_{i\leq N}W_i(n)=\max_{i\leq N}\sup_{0\leq k\leq n}\sum_{j=k+1}^n (A_{i,j}B_j-T_j).
\end{align} 

We assume that the sequence of random variables $(B_j,j\geq 1)$ are independent random variables that satisfy 
\begin{align}\label{eq: def regularly varying rv}
    \probability{B_j>x}= L(x)/x^{\beta},
\end{align}
 with $L(x)$ a slowly varying function and $\beta>1$, indicating possible large job sizes. We let i.i.d.\ random variables $(A_{i,j},i\geq 1, j\geq 1)$ satisfy 
\begin{align}\label{eq: asymptotic relation weibull}
\log \probability{A_{i,j}>x}\sim -qx^{\alpha},
\end{align}
as $x\to\infty$, with $0<\alpha<1$ and $q>0$. Let $b_N=\left(\frac{\log N}{q}\right)^{1/\alpha}$. Then, we know from standard extreme value theory \cite[Thm.\ 5.4.1, p.\ 188]{de2007extreme} that
\begin{align}
\frac{\max_{i\leq N}A_i}{b_N}\LimitP 1,
\end{align}
as $N\to\infty$. Thus, the number $b_N$ indicates the approximate size of the largest of $N$ independent Weibull-distributed random variables. Furthermore, we have independent and identically distributed distributed random variables $(T_j,j\geq 1)$, such that 
\begin{align}
\mathbb{E}[A_{i,j}B_j-T_j]=-\mu,
\end{align}
with $\mu>0$.

 In this paper, we prove process convergence of the scaled maximum waiting time over $N$ servers in Theorem \ref{thm: main result waiting time heavy tail}; cf.\ \cite{scholMOR} for a similar result for fork-join queues with light-tailed services. In order to achieve this result, we need to scale the number of arriving jobs and the maximum waiting time with a sequence $(c_N,N\geq 1)$, where the sequence $(c_N,N\geq 1)$ satisfies
\begin{align}\label{eq: property cN}
 c_N\sim\frac{(c_N/b_N)^{\beta}}{L(c_N/b_N)},   
\end{align} 
as $N\to\infty$, with $\frac{c_N}{b_N}\LimitN \infty$. We explain in Section \ref{subsec: heuristic analysis} in more detail why this sequence scales as given in \eqref{eq: property cN}.  Following standard arguments on generalized inverses of regularly varying functions; cf.\ \cite[Prop.\ 2.6 (v,vi,vii)]{resnick2007heavy} and \cite[Thm.\ 1.5.12]{bingham1989regular}, we can solve the right-hand side of \eqref{eq: property cN} and get that $\frac{c_N}{b_N}\sim c_N^{\frac{1}{\beta}}/\hat{L}(c_N)$, with $\hat{L}$ a slowly varying function. From this, it follows that $b_N\sim\hat{L}(c_N)c_N^{\frac{(\beta-1)}{\beta}}$. Now, we define the sequence $(c_N,N\geq 1)$ as 
\begin{align}\label{eq: def cN}
    c_N:=\tilde{L}(b_N)b_N^{\frac{\beta}{(\beta-1)}}
\end{align}
where $\tilde{L}$ is a slowly varying function that equals
\begin{align}\label{eq: tilde L descr}
    \tilde{L}(x)x^{\frac{\beta}{(\beta-1)}}=\left(\left(\left(\frac{x}{\left(\frac{x^\beta}{L(x)}\right)^{\leftarrow}}\right)^*\right)^{\leftarrow}\right)^*,
\end{align}
with $H(y)^{\leftarrow}=\inf\{s: H(s)\geq y\}$, and $f(x)^*$ is a monotone function with the property that $f(x)^*\sim f(x)$ as $x\to\infty$. Thus, the sequence $(c_N,N\geq 1)$ satisfies the relation described in \eqref{eq: property cN}. More precise properties of the function $\tilde{L}$ are given in Lemma \ref{lem: properties tilde L}.

As we have a proper scaling of the number of arriving jobs and the maximum waiting time by a sequence $(c_N,N\geq 1)$, the scaled maximum waiting time has the form 
\begin{align}\label{eq: waiting time}
    \frac{\max_{i\leq N}W_i(tc_N)}{c_N}=\sup_{s\in[0,t]}\frac{\max_{i\leq N}\sum_{j=\lfloor sc_N\rfloor+1}^{\lfloor tc_N\rfloor}(A_{i,j}B_j-T_j)}{c_N}.
\end{align}
Notice that 
\begin{align}\label{eq: equality in dist max waiting time}
\sup_{s\in[0,t]}\frac{\max_{i\leq N}\sum_{j=\lfloor sc_N\rfloor+1}^{\lfloor tc_N\rfloor}(A_{i,j}B_j-T_j)}{c_N}\overset{d}{=} \sup_{s\in[0,t]}\frac{\max_{i\leq N}\sum_{j=1}^{\lfloor sc_N\rfloor}(A_{i,j}B_j-T_j)}{c_N}.
\end{align}
Thus, to prove convergence of a single random variable $\frac{\max_{i\leq N}W_i(tc_N)}{c_N}$ it suffices to prove convergence of the right-hand side in Equation \eqref{eq: equality in dist max waiting time}. However, the processes $(\sup_{s\in[0,t]}\frac{\max_{i\leq N}\sum_{j=1}^{\lfloor sc_N\rfloor}(A_{i,j}B_j-T_j)}{c_N},t\in[0,T])$ and $(\frac{\max_{i\leq N}W_i(tc_N)}{c_N},t\in[0,T])$ are not equal in distribution. For instance, $(\sup_{s\in[0,t]}\frac{\max_{i\leq N}\sum_{j=1}^{\lfloor sc_N\rfloor}(A_{i,j}B_j-T_j)}{c_N},t\in[0,T])$, which we will refer to as the \textit{auxiliary process}, is non-decreasing in $t$, and $(\frac{\max_{i\leq N}W_i(tc_N)}{c_N}\allowbreak,t\in[0,T])$ is not non-decreasing in $t$. In Theorem \ref{thm: auxiliary result waiting time heavy tail}, we show that this auxiliary process converges in distribution to a limiting process;
$$
\left(\sup_{s\in[0,t]}\frac{\max_{i\leq N}\sum_{j=1}^{\lfloor sc_N\rfloor} (A_{i,j}B_j-T_j)}{c_N},t\in[0,T]\right)\LimitD \left(\sup_{s\in[0,t]}(X_s-\mu s),t\in[0,T]\right),
$$
as $N\to\infty$. The process $(X_t,t\in[0,T])$ is a stochastic process with Fr\'echet-distributed marginals. This process has cumulative distribution function $\probability{X_t\leq x}=\exp(-t/x^{\beta})$ for $x>0$. Furthermore, $X_{t+s}=\max(X_t,\hat{X}_s)$, where $\hat{X}_s$ is an independent copy of $X_s$, because $\probability{X_{t+s}<x}=\probability{X_t<x}\probability{\hat{X}_s<x}=\exp(-t/x^{\beta})\exp(-s/x^{\beta})=\exp(-(t+s)/x^{\beta})$. Thus, the process $(X_t,t\in[0,T])$ is a function in $D[0,T]$ and is called an extremal process, cf.\ \cite{resnick1973structure}. It is easy to see that $(\sup_{s\in[0,t]}(X_s-\mu s),t\in[0,T])$ is also non-decreasing in $t$. The limiting process of $(\frac{\max_{i\leq N}W_i(tc_N)}{c_N},t\in[0,T])$ has the same marginals as the process $(\sup_{s\in[0,t]}(X_s-\mu s),t\in[0,T])$, but is not non-decreasing. We write the limiting process of the maximum waiting time as $(\sup_{s\in[0,t]}(X_{(s,t)}-\mu(t-s)),t\in[0,T])$, with $X_{(s,t)}\overset{d}{=}X_{t-s}$. For $r<s<t$, we have that $X_{(r,t)}=\max(X_{(r,s)},X_{(s,t)})$, and we have that $X_{(s,t)}$ and $X_{(u,v)}$ are independent if and only if the intervals $(s,t)$ and $(u,v)$ are disjoint. We write $X_t=X_{(0,t)}$. In conclusion, the random variable $X_t$ involves a single time parameter, while the random variable $X_{(s,t)}$ is defined by two time parameters, which complicates the proof. There is a clear connection between the stochastic processes however, and in this paper, we first prove convergence of the non-decreasing process $(\frac{\max_{i\leq N}\sup_{s\in[0,t]}\sum_{j=1}^{\lfloor sc_N\rfloor} (A_{i,j}B_j-T_j)}{c_N},t\in[0,T])$ and we use this result with some additional steps to prove process convergence of the scaled maximum waiting time $(\frac{\max_{i\leq N}W_i(tc_N)}{c_N},t\in[0,T])$.
\begin{assumption}[Waiting time]\label{ass: waiting time}
Let the waiting time of customers in front of the $i$-th server be given in Equation \eqref{eq: indiv waiting time}, the i.i.d.\ random variables $(A_{i,j},i\geq 1,j\geq 1)$ satisfy \eqref{eq: asymptotic relation weibull}, and the i.i.d.\ random variables $(B_j,j\geq 1)$ satisfy \eqref{eq: def regularly varying rv} with $L(x)$ a slowly varying function. 
\end{assumption}
\begin{assumption}[Scaling]\label{ass: scaling}
Let $b_N=(\log N/q)^{\frac{1}{\alpha}}$, and $(c_N,N\geq 1)$ and $\tilde{L}$ satisfy \eqref{eq: property cN}, \eqref{eq: def cN}, and \eqref{eq: tilde L descr}.
\end{assumption}
\begin{assumption}[Limiting process]\label{ass: limiting process}
 Let $(X_{(s,t)},t\in[0,T])$ be a stochastic process with Fr\'echet-distributed marginals. For $r<s<t$, we have that $X_{(r,t)}=\max(X_{(r,s)},X_{(s,t)})$, and we have that $X_{(s,t)}$ and $X_{(u,v)}$ are independent if and only if the intervals $(s,t)$ and $(u,v)$ are disjoint. We write $X_t=X_{(0,t)}$ and we have that $X_{(s,t)}\overset{d}{=}X_{t-s}$. Furthermore, $\probability{X_t\leq x}=\exp(-t/x^{\beta})$ for $x>0$.
\end{assumption}
\begin{theorem}\label{thm: auxiliary result waiting time heavy tail}
Given that Assumptions \ref{ass: waiting time}--\ref{ass: limiting process} hold, we have that 
\begin{align}\label{eq: time convergence auxiliary process}
\left(\sup_{s\in[0,t]}\frac{\max_{i\leq N}\sum_{j=1}^{\lfloor sc_N\rfloor} (A_{i,j}B_j-T_j)}{c_N},t\in[0,T]\right)\LimitD \left(\sup_{s\in[0,t]}(X_s-\mu s),t\in[0,T]\right),
\end{align}
as $N\to\infty$.
\end{theorem}

Now, the main result proven in this paper is given in Theorem \ref{thm: main result waiting time heavy tail}.
\begin{theorem}\label{thm: main result waiting time heavy tail}
Given that Assumptions \ref{ass: waiting time}--\ref{ass: limiting process} hold, we have that 
\begin{align}\label{eq: conv in d main result}
\left(\frac{\max_{i\leq N}W_i(tc_N)}{c_N},t\in[0,T]\right)\LimitD \left(\sup_{s\in[0,t]}(X_{(s,t)}-\mu(t-s)),t\in[0,T]\right),
\end{align}
as $N\to\infty$.
\end{theorem}
Now, when letting $t\to\infty$ on the left-hand and the right-hand side of \eqref{eq: conv in d main result}, we expect from this convergence result that the maximum steady-state waiting time satisfies $\probability*{\max_{i\leq N}W_i(\infty)>xc_N}\LimitN \probability*{\sup_{t>0}(X_t-\mu t)>x}$. Though this does not trivially follow from Theorem \ref{thm: main result waiting time heavy tail}, it is indeed true, and we prove this in Theorem \ref{thm: steady state heavy tail}.
\begin{theorem}\label{thm: steady state heavy tail}
Given that Assumptions \ref{ass: waiting time}--\ref{ass: limiting process} hold, we have that
\begin{align}\label{eq: conv in d steady state}
   \probability*{\max_{i\leq N}W_i(\infty)>xc_N}\LimitN \probability*{\sup_{t>0}(X_t-\mu t)>x}.
\end{align}
\end{theorem}
We can write the limiting probabilities explicitly.
\begin{proposition}\label{prop: marginal steady state}
Given that Assumptions \ref{ass: waiting time}--\ref{ass: limiting process} hold, we have that
\begin{align}\label{eq: steady state explicit prob}
\probability*{\sup_{t>0}(X_t-\mu t)>x}= 1-\exp\left(-\frac{1}{\mu(\beta-1)x^{\beta-1}}\right),
\end{align}
and
\begin{align}\label{eq: transient explicit prob}
    \probability*{\sup_{s\in[0,t]}(X_{(s,t)}-\mu (t-s))>x}=1-\exp\left(-\frac{1}{\mu^{\beta}(\beta-1)}\left(\frac{1}{(x/\mu)^{\beta-1}}-\frac{1}{(x/\mu+t)^{\beta-1}}\right)\right).
\end{align}
\end{proposition}
\subsection{Main ideas for the proofs}\label{subsec: heuristic analysis}
To prove Theorem \ref{thm: main result waiting time heavy tail} directly is challenging, since the limiting random variable $X_{(s,t)}$ depends on two parameters and cannot be written as a difference of the form $Y_t-Y_s$, as is the case in standard queueing theory. However, the marginal distributions of $X_{(s,t)}$ and $X_{t-s}$ are the same. Thus, we first prove Theorem \ref{thm: auxiliary result waiting time heavy tail}, after which we prove Theorem \ref{thm: main result waiting time heavy tail} using some auxiliary results on bounds on tail probabilities, convergence rates of sums of Weibull-distributed random variables, and auxiliary results on process convergence in $D[0,T]$; cf.\ Section \ref{sec: prel results heavy tails}. To get a better understanding of the convergence result in Theorem \ref{thm: auxiliary result waiting time heavy tail}, it benefits to first examine the process
\begin{align}
\left(\frac{\max_{i\leq N}\sum_{j=1}^{\lfloor tc_N\rfloor} (A_{i,j}B_j-T_j)}{c_N},t\in[0,T]\right),
\end{align}
so we remove the supremum term from the expression on the left-hand side of \eqref{eq: time convergence auxiliary process} and we are left with a maximum of $N$ random walks. We can however apply the continuous mapping theorem on this stochastic process and obtain the result in Theorem \ref{thm: auxiliary result waiting time heavy tail}, because the supremum is a continuous functional. Obviously, 
the law of large numbers implies that
\begin{align}
   \frac{\sum_{j=1}^{\lfloor tc_N\rfloor} (A_{i,j}B_j-T_j)}{c_N} \LimitP -\mu t,
\end{align}
as $N\to\infty$. However, when we investigate the largest of $N$ of these random variables, we obtain that 
\begin{align}
   \frac{\max_{i\leq N}\sum_{j=1}^{\lfloor tc_N\rfloor} (A_{i,j}B_j-T_j)}{c_N} \LimitD X_t-\mu t,
\end{align}
as $N\to\infty$. The fact that we see this limiting behavior has two main reasons; first of all, a standard result is that for i.i.d.\ regularly varying $(B_j,j\geq 1)$, the tail behavior of a finite sum is the same as the tail behavior of the largest regularly varying random variable. Second, for Weibull-distributed random variables and a deterministic sequence $(b_j,j\geq 1)$, we have that $\max_{i\leq N}\sum_{j=1}^n A_{i,j}b_j/b_N\LimitP \max_{j\leq n}b_j$, as $N\to\infty$, cf.\ \cite[Lem.\ B.1]{scholMOR} for a proof. Therefore, $\max_{i\leq N}\sum_{j=1}^nA_{i,j}B_j\approx \max_{i\leq N}A_i\cdot\max_{j\leq n}B_j+\mathbb{E}[A_{i,j}B_j](n-1)$ for $N$ large. Thus, we can conclude that for $N$ large,
\begin{align}
    \frac{\max_{i\leq N}\sum_{j=1}^{\lfloor tc_N\rfloor} (A_{i,j}B_j-T_j)}{c_N}&\approx \frac{\max_{i\leq N}A_i}{b_N}\frac{\max_{j\leq \lfloor tc_N\rfloor}B_j}{\frac{c_N}{b_N}}+  \frac{\sum_{j=1}^{\lfloor tc_N\rfloor} (A_{i,j}B_j-T_j)}{c_N}\label{subeq: approx 1}\\
   &\approx \frac{\max_{i\leq N}A_i}{b_N}\frac{\max_{j\leq \lfloor tc_N\rfloor}B_j}{\frac{c_N}{b_N}}-\mu t\label{subeq: approx 3}\\
 &  \approx \frac{\max_{j\leq \lfloor tc_N\rfloor}B_j}{\frac{c_N}{b_N}}-\mu t.\label{subeq: approx 4}
\end{align}
We see that the largest regularly varying random variable $\max_{j\leq \lfloor tc_N\rfloor}B_j$ determines the stochastic part in the limit, and is of order $\frac{c_N}{b_N}$. Now, it is easy to see that
\begin{align*}
    \probability*{\max_{j\leq \lfloor tc_N\rfloor}B_j\leq (x+\mu t)\frac{c_N}{b_N}}=\probability*{B_j\leq (x+\mu t)\frac{c_N}{b_N}}^{\lfloor tc_N\rfloor}\sim \left(1-\frac{L((x+\mu t)\frac{c_N}{b_N})}{\left((x+\mu t)\frac{c_N}{b_N}\right)^{\beta}}\right)^{\lfloor tc_N\rfloor}.
\end{align*}
Because we have defined $c_N$ as having the relation $c_N\sim\frac{(c_N/b_N)^{\beta}}{L(c_N/b_N)}$ as $N\to\infty$, we get that, 
$$
\left(1-\frac{L((x+\mu t)\frac{c_N}{b_N})}{\left((x+\mu t)\frac{c_N}{b_N}\right)^{\beta}}\right)^{\lfloor tc_N\rfloor}\sim \left(1-\frac{1}{(x+\mu t)^{\beta}c_N}\right)^{\lfloor tc_N\rfloor}\LimitN \exp\left(-\frac{t}{(x+\mu t)^{\beta}}\right).
$$
In conclusion, the limiting distribution of $\max_{i\leq N}\sum_{j=1}^{\lfloor tc_N\rfloor} (A_{i,j}B_j-T_j)/c_N$ is a Fr\'echet-distributed random variable with a negative drift term. We also see that we can approximate $\max_{i\leq N}\sum_{j=1}^{\lfloor tc_N\rfloor} (A_{i,j}B_j-T_j)/c_N$ with $\max_{j\leq \lfloor tc_N\rfloor}B_j/(\frac{c_N}{b_N})-\mu t$ as $N$ is large. This approximating random variable has convenient properties, since the stochastic term is non-decreasing in $t$. Therefore, to prove process convergence of\\ $\left(\max_{i\leq N}\sum_{j=1}^{\lfloor tc_N\rfloor} (A_{i,j}B_j-T_j)/c_N,t\in[0,T]\right)$ to $(X_t-\mu t,t\in[0,T])$, we first prove that $(\max_{j\leq \lfloor tc_N\rfloor}B_j/(\frac{c_N}{b_N})-\mu t,t\in[0,T])$ converges to $(X_t-\mu t,t\in[0,T])$. Furthermore, we prove in Lemma \ref{lem: sup max B and process} that for all $\epsilon>0$,
$$
\probability*{\sup_{t\in[0,T]}\left|\frac{\max_{i\leq N}\sum_{j=1}^{\lfloor tc_N\rfloor}(A_{i,j}B_j-T_j)}{c_N}-\left(\frac{\max_{j\leq \lfloor tc_N\rfloor}B_j}{\frac{c_N}{b_N}}-\mu t\right)\right|>\epsilon}\LimitN 0.
$$
After applying the triangle inequality, we obtain that 
$$
\left(\frac{\max_{i\leq N}\sum_{j=1}^{\lfloor tc_N\rfloor} (A_{i,j}B_j-T_j)}{c_N},t\in[0,T]\right)\LimitD \left(X_t-\mu t,t\in[0,T]\right),
$$
as $N\to\infty$. Now, by applying the continuous mapping theorem, we obtain the result of Theorem \ref{thm: auxiliary result waiting time heavy tail}. This is still an auxiliary result, because the process on the left-hand side of \eqref{eq: time convergence auxiliary process} is not the maximum waiting time process. We can however prove the process convergence of the maximum waiting time in Theorem \ref{thm: main result waiting time heavy tail} by using some additional results, as the marginals of the processes on the left side of the limit in Theorems \ref{thm: auxiliary result waiting time heavy tail} and \ref{thm: main result waiting time heavy tail} are the same, and the marginals of the limiting processes in Theorems \ref{thm: auxiliary result waiting time heavy tail} and \ref{thm: main result waiting time heavy tail} are the same. Thus, we already know that pointwise convergence holds.

In order to prove the convergence of the finite-dimensional distributions for the maximum waiting time process, we show that we can decompose the joint probabilities of both the maximum waiting time process and the limiting process into an operation of marginal probabilities, and thus, convergence of finite-dimensional distributions follows from pointwise convergence. For example, for $x_2+\mu t_2>x_1+\mu t_1$,
\begin{multline*}
\probability*{\sup_{s\in[0,t_1]}(X_{(s,t_1)}-\mu(t_1- s))<x_1\cap \sup_{s\in[0,t_2]}(X_{(s,t_2)}-\mu(t_2- s))<x_2}\\
=\frac{\probability*{\sup_{s\in[0,t_1]}(X_{(s,t_1)}-\mu(t_1- s))<x_1}}{\probability*{\sup_{s\in[0,t_1]}(X_{(s,t_1)}-\mu(t_1- s))<x_2+\mu(t_2-t_1)}}\probability*{\sup_{s\in[0,t_2]}(X_{(s,t_2)}-\mu(t_2-s))<x_2}.
\end{multline*}
An analogous equation holds for the process $(\sup_{s\in[0,t]}(\max_{\lfloor sc_N\rfloor\leq j\leq \lfloor tc_N\rfloor }\frac{B_j}{(\frac{c_N}{b_N})}-\mu(t- s)),t\in[0,T])$.

In Lemma \ref{lem: sup max B and process2}, we prove that the maximum waiting time in \eqref{eq: waiting time} satisfies
\begin{align*}
\probability*{\sup_{t\in[0,T]}\bigg|\sup_{s\in[0,t]}\frac{\max_{i\leq N}\sum_{j=\lfloor sc_N\rfloor+1}^{\lfloor tc_N\rfloor}(A_{i,j}B_j-T_j)}{c_N}-\sup_{s\in[0,t]}\left(\frac{\max_{\lfloor sc_N\rfloor\leq j\leq \lfloor tc_N\rfloor}B_j}{\frac{c_N}{b_N}}-\mu (t-s)\right)\bigg|>\epsilon}\LimitN 0,
\end{align*}
by using similar techniques as in Lemma \ref{lem: sup max B and process}. Finally, we show in the proof of Theorem \ref{thm: main result waiting time heavy tail} that 
$$
\left(\sup_{s\in[0,t]}\left(\max_{\lfloor sc_N\rfloor\leq j\leq \lfloor tc_N\rfloor }\frac{B_j}{(\frac{c_N}{b_N})}-\mu(t- s)\right),t\in[0,T]\right)\LimitD \left(\sup_{s\in[0,t]}(X_{(s,t)}-\mu(t-s)),t\in[0,T]\right),$$ 
as $N\to\infty$, by using the earlier results together with Lemma \ref{lem: billingsley conv in D} from \cite[Thm.\ 13.3]{billingsley2013convergence}.

In summary, we prove process convergence of the maximum waiting time through three steps; first, pointwise convergence follows from Theorem \ref{thm: auxiliary result waiting time heavy tail}; second, we show in Lemma \ref{lem: sup max B and process2} that the maximum waiting process is asymptotically equivalent to an extremal process that only depends on the regularly varying random variables, and finally, we prove process convergence for this latter process in Theorem \ref{thm: main result waiting time heavy tail}. 

In Section \ref{sec: steady state}, we show that the cumulative distribution function of the limiting maximum steady-state waiting time converges to $\probability{\sup_{t>0}(X_t-\mu t)<x}$. This means that the limiting cumulative distribution function of the maximum steady-state waiting time is the same as $\lim_{T\to\infty}\probability{\sup_{s\in[0,T]}(X_{(s,T)}-\mu(T-s))<x}$, thus the steady-state behavior of the limiting process of $(\sup_{s\in[0,t]}(X_{(s,t)}-\mu(t-s)),t\in[0,T])$ is the same as the extreme-value limit of the maximum steady-state waiting time, which is not a trivial result.

\subsection{Other choices for $A_{i,j}$}\label{subsec: other choices}
Our main result in Theorem \ref{thm: main result waiting time heavy tail} heavily relies on the fact that $A_{i,j}$ is Weibull-like, and we are able to derive general results. Furthermore, when $A_{i,j}$ has finite support, we are also able to derive general results. Under the assumption that $A_{i,j}$ has a finite right endpoint $b$, it follows that $b_N=b$. Furthermore, in \cite[Lem.\ B.1]{scholMOR}, we have shown that $\max_{i\leq N}\sum_{j=1}^nx_jA_{i,j}/b\LimitP \sum_{j=1}^nx_j$, as $N\to\infty$. Then,
\begin{align}
\probability*{\max_{i\leq N}\sum_{j=1}^k(A_{i,j}B_j-T_j)>x}\LimitN \probability*{\sum_{j=1}^k(bB_j-T_j)>x}.
\end{align}
Furthermore, if $\mathbb{E}[bB_j-T_j]<0$, we get that 
\begin{align}
\probability*{\max_{i\leq N}\sup_{k\geq 0}\sum_{j=1}^k(A_{i,j}B_j-T_j)>x}\LimitN\probability*{\sup_{k\geq 0}\sum_{j=1}^k(bB_j-T_j)>x}.
\end{align}

In general, when $A_{i,j}$ has unbounded support, from \cite[Lem.\ B.1]{scholMOR} follows that
\begin{align}
\probability*{\max_{i\leq N}\sup_{0\leq k\leq l}\sum_{j=1}^k(A_{i,j}B_j-T_j)>xb_N}\LimitN \probability*{\max_z\left\{\sum_{j=1}^lz_{j}B_j\Bigg|\sum_{j=1}^lz_j^{\alpha}\leq 1,0\leq z_j\leq 1\right\}>x}.
\end{align} 
For the heavy-tailed case, i.e., when $0<\alpha\leq 1$, we have that $\max_z\left\{\sum_{j=1}^lz_{j}B_j\Big|\sum_{j=1}^lz_j^{\alpha}\leq 1,0\leq z_j\leq 1\right\}=\max_{j\leq l}B_j$. In contrary, for the light-tailed case, i.e., when $\alpha>1$, we get nontrivial results depending on $\alpha$. For instance for $\alpha=2$, we obtain $\max_z\left\{\sum_{j=1}^lz_{j}B_j\Big|\sum_{j=1}^lz_j^{\alpha}\leq 1,0\leq z_j\leq 1\right\}=\sqrt{B_1^2+\cdots+B_l^2}$. In conclusion, when $\alpha> 1$, we cannot rely on the property described in \eqref{subeq: approx 1}--\eqref{subeq: approx 4} that follows from the behavior of Weibull-distributed random variables. In this paper, we limit ourselves to the case $0<\alpha<1$; the case $\alpha>1$ falls outside the scope of this paper. The case $\alpha=1$ lies on the boundary between these two regimes; this case needs a separate analysis. 

\section{Preliminary results}\label{sec: prel results heavy tails}
In Section \ref{subsec: heuristic analysis}, we gave heuristic ideas of our results. In order to be able to prove these, we need some auxiliary lemmas.

In \eqref{eq: property cN}, \eqref{eq: def cN}, and \eqref{eq: tilde L descr}, we heuristically describe the behavior of the sequence $(c_N,N\geq 1)$ and the slowly varying function $\tilde{L}$ given a sequence $(b_N,N\geq 1)$ and a slowly varying function $L$. An unanswered question is whether this sequence $(c_N,N\geq 1)$ and this function $\tilde{L}$ exists. In Lemma \ref{lem: properties tilde L}, we show how, if this $\tilde{L}$ exists, $L$ and $\tilde{L}$ are asymptotically related. Their asymptotic relation resembles the asymptotic relation between a slowly varying function $l$ and its de Bruijn conjugate $l^{\#}$, cf.\ \cite[Thm.\ 1.5.13]{bingham1989regular}. The proof of the existence of $\tilde{L}$ is analogous to the proof of existence of $l^{\#}$ given in \cite[Thm.\ 1.5.13]{bingham1989regular}, thus we omit it here.
\begin{lemma}[Asymptotic behavior of $\tilde{L}(x)$]\label{lem: properties tilde L}
From \eqref{eq: property cN}, \eqref{eq: def cN}, and \eqref{eq: tilde L descr} follows that the function $\tilde{L}$ satisfies the relation
\begin{align}\label{eq: asymptotic relation tilde L}
    \tilde{L}(x)\sim L(\tilde{L}(x)x^{\frac{1}{(\beta-1)}})^{\frac{1}{(\beta-1)}},
\end{align}
as $x\to\infty$.
\end{lemma}
\begin{proof}
We write $x=b_N$, then the relation in \eqref{eq: property cN} can be rewritten to
$$
\tilde{L}(x)x^{\frac{\beta}{(\beta-1)}}\sim \frac{(\tilde{L}(x)x^{\frac{\beta}{(\beta-1)}}/x)^{\beta}}{L(\tilde{L}(x)x^{\frac{\beta}{(\beta-1)}}/x)},
$$
as $x\to\infty$. This simplifies to
$$
\tilde{L}(x)\sim \frac{\tilde{L}(x)^{\beta}}{L(\tilde{L}(x)x^{\frac{1}{(\beta-1)}})},
$$
as $x\to\infty$. The lemma follows.
\end{proof}
 \begin{remark}[Asymptotic solutions of $\tilde{L}(x)$]
 It is not trivial to find functions $\tilde{L}(x)$ that have the asymptotic relation described in \eqref{eq: asymptotic relation tilde L}, since $\tilde{L}$ appears both on the left and the right side of the equation. However, we know that $\tilde{L}$ is slowly varying, thus the term $x^{\frac{1}{(\beta-1)}}$ is dominant in $L(\tilde{L}(x)x^{\frac{1}{(\beta-1)}})^{\frac{1}{(\beta-1)}}$, so we can remove $\tilde{L}$ from the right-hand side in \eqref{eq: asymptotic relation tilde L} and look at the function $\tilde{L}^{(1)}$ that equals
 $$
 \tilde{L}^{(1)}(x)=L(x^{\frac{1}{(\beta-1)}})^{\frac{1}{(\beta-1)}}.
 $$
 For example, when $L(x)=\log x$, $\tilde{L}^{(1)}$ satisfies \eqref{eq: asymptotic relation tilde L}. However, there are also examples where $\tilde{L}^{(1)}$ does not satisfy the relation in \eqref{eq: asymptotic relation tilde L}, for example, when $L(x)=\exp(\sqrt{\log x})$. Still, we are able to find candidates that satisfy the relation in \eqref{eq: asymptotic relation tilde L}. First, we see that the relation in \eqref{eq: asymptotic relation tilde L} is actually an iterative relation. Thus, we can rewrite \eqref{eq: asymptotic relation tilde L} to 
 $$
 \tilde{L}(x)\sim L(L(\tilde{L}(x)x^{\frac{1}{(\beta-1)}})^{\frac{1}{(\beta-1)}}x^{\frac{1}{(\beta-1)}})^{\frac{1}{(\beta-1)}},
 $$
as $x\to\infty$. Now, with the same reasoning as before, we define 
$$
\tilde{L}^{(2)}(x)=L(L(x^{\frac{1}{(\beta-1)}})^{\frac{1}{(\beta-1)}}x^{\frac{1}{(\beta-1)}})^{\frac{1}{(\beta-1)}}.
$$
The function $\tilde{L}^{(2)}$ satisfies the relation in \eqref{eq: asymptotic relation tilde L} when $L(x)=\exp(\sqrt{\log x})$ and is a slowly varying function itself.
 \end{remark}
In order to prove that the heuristic approximations in Equations \eqref{subeq: approx 1}--\eqref{subeq: approx 4} are correct, we need to prove two things; first, that the largest regularly varying random variable determines the stochastic part of the limit, and second, that the other random variables satisfy the law of large numbers. To prove this second property, we use Bennett's inequality as stated below. In Corollary \ref{cor: bennett}, we state a simplified version of this inequality which we use in our proofs.
\begin{lemma}[Bennett's inequality \cite{bennett1962probability}]\label{lem: bennett}
Let $Y_1,\ldots,Y_n$ be independent random variables, $\mathbb{E}[Y_i]=0$, $\mathbb{E}[Y_i^2]=\sigma_i^2$, and $|Y_i|<M\in\mathbb{R}$ almost surely. Then for $y>0$,
$$
\probability*{\sum_{i=1}^nY_i>y}\leq \exp\left(-\frac{\sum_{i=1}^n\sigma_i^2}{M^2}h\left(\frac{yM}{\sum_{i=1}^n\sigma_i^2}\right)\right),
$$
with $h(x)=(1+x)\log(1+x)-x$.
\end{lemma}
For a proof, cf.\ \cite{zheng2018improved}.
\begin{corollary}\label{cor: bennett}
Let $Y_1,\ldots,Y_n$ be independent random variables, $\mathbb{E}[Y_i]=0$, $\mathbb{E}[Y_i^2]=\sigma_i^2$, and $|Y_i|<M$ almost surely. Then for $y>0$,
$$
\probability*{\sum_{i=1}^nY_i>y}\leq \exp\left(-\frac{y}{M}\left(\log\left(1+\frac{yM}{\sum_{i=1}^n\sigma_i^2}\right)-1\right)\right).
$$
\end{corollary}
\begin{proof}
Observe that for $x>0$ we get that $h(x)>x(\log(1+x)-1)$. Now, the corollary follows from Lemma \ref{lem: bennett}.
\end{proof}
Though in \cite[Lem.\ B.1]{scholMOR} it is proven that for Weibull-distributed random variables $\max_{i\leq N}\sum_{j=1}^nA_{i,j}b_j/b_N\LimitP \max_{j\leq n}b_j$, as $N\to\infty$, which heuristically explains the nature of our main result, our approximations in Equations \eqref{subeq: approx 1}--\eqref{subeq: approx 4} suggest that we should take the sum of $\lfloor tc_N\rfloor$ random variables. In \cite[Lem.\ B.1]{scholMOR} however, $n$ does not depend on $N$. Thus, we cannot resort to \cite[Lem.\ B.1]{scholMOR} in our proofs. However, in Lemma \ref{lem: large deviations sum}, a result is presented that we can use in this paper and proves the approximations in Equations \eqref{subeq: approx 1}--\eqref{subeq: approx 4}.
\begin{lemma}[{{\cite[Thm.\ 2]{brosset2022large}}}]\label{lem: large deviations sum}
Let $Y_1,\ldots,Y_n$ be independent random variables, with $\log\probability{Y_i>x}\sim -qx^{\alpha}$, as $x\to\infty$, with $0<\alpha<1$ and $q>0$. Let $(x_n,n\geq 1)$ be a sequence such that $\lim_{n\to\infty}x_n/n^{\frac{1}{(2-\alpha)}}=\infty$. Then 
$$
\lim_{n\to\infty}\frac{1}{x_n^{\alpha}}\log\probability*{\sum_{i=1}^nY_i>x_n}=-q.
$$
\end{lemma}
We want to prove process convergence of the maximum waiting time to a limiting process, this limiting process is a function in $D[0,T]$. In \cite[Thm.\ 13.3]{billingsley2013convergence}, a result is given that guarantees the convergence of a process in $D[0,T]$ when three conditions are satisfied, which we will apply in this paper.
\begin{lemma}[{{\cite[Thm.\ 13.3]{billingsley2013convergence}}}]\label{lem: billingsley conv in D}
Assume a sequence of processes $(Y^{(N)}(t),t\in[0,T])$ and a process $(Y(t),t\in[0,T])$ in $D[0,T]$, equipped with the $d^0$ metric, satisfy the following conditions:
\begin{enumerate}
    \item For all $\{t_1,\ldots,t_k\}\subseteq [0,T]$: $(Y^{(N)}(t_1),\ldots,Y^{(N)}(t_k))\LimitD(Y(t_1),\ldots,Y(t_k))$ as $N\to\infty$.
    \item $Y(T)-Y(T-\delta)\overset{\mathbb{P}}{\longrightarrow}0$ as $\delta\downarrow 0$,
    and
    \item For $0<r<s<t<T$, $\epsilon,\eta>0$ there exists $N_0\geq 1$ and $\delta>0$ such that 
    $$\probability*{\sup_{s\in[r,t],t-r<\delta}\min\left(\bigg|Y^{(N)}(s)-Y^{(N)}(r)\bigg|,\bigg|Y^{(N)}(t)-Y^{(N)}(s)\bigg|\right)>\epsilon}\leq \eta$$
\end{enumerate}
for $N\geq N_0$.
Then $(Y^{(N)}(t),t\in[0,T])\LimitD (Y(t),t\in[0,T])$ as $N\to\infty$.
\end{lemma}
Finally, to prove pointwise convergence of the maximum waiting time process in \eqref{eq: conv in d main result} to the limiting random variable, we need to pay special attention to the case that $B$ is a regularly varying random variable with $1<\beta\leq 2$, since in this case the second moment of $B$ is not finite. In Lemma \ref{lem: second moment alpha small}, we give a useful convergence result of the second moment of $B$ conditioned on $B$ being bounded. 
\begin{lemma}\label{lem: second moment alpha small}
Let $B$ be a positive random variable that satisfies $\probability{B>x}=L(x)/x^{\beta}$, with $L(x)$ a slowly varying function and $1<\beta\leq 2$. Then, 
$$
\frac{\mathbb{E}[B^2|B<r]}{r}\to 0,
$$
as $r\to\infty$.
\end{lemma}
\begin{proof}
Choose $0<\epsilon<\beta-1$. Because $\probability{B>x}= L(x)/x^{\beta}$, we have that $\mathbb{E}[B^{\beta-\epsilon}]<\infty$. Therefore,
$$
\frac{\mathbb{E}[B^2|B<r]}{r}\leq\frac{r^{2-(\beta-\epsilon)}}{r}\mathbb{E}[B^{\beta-\epsilon}]\to 0,
$$
as $r\to\infty$.
\end{proof}
\section{Convergence of the auxiliary process in $D[0,T]$}\label{sec: conv in D}\label{sec: process convergence}
In this section, we prove Theorem \ref{thm: auxiliary result waiting time heavy tail}. As explained in Section \ref{subsec: heuristic analysis}, we first remove the supremum functional from the random variable on the left-hand side in \eqref{eq: time convergence auxiliary process} and prove convergence of the process $(\max_{i\leq N}\sum_{j=1}^{\lfloor tc_N\rfloor} (A_{i,j}B_j-T_j)/c_N,t\in[0,T])$ to $(X_t-\mu t,t\in[0,T])$. To do so, we first show pointwise convergence in Lemma \ref{lem: pointwise convergence}; afterwards we prove process convergence in Lemma \ref{lem: conv in D sum}. In order to prove Lemma \ref{lem: conv in D sum}, we need two auxiliary results, which are given in Lemmas \ref{lem: conv in D max B} and \ref{lem: sup max B and process}. By using the continuous mapping theorem, Theorem \ref{thm: auxiliary result waiting time heavy tail} follows. 
\begin{lemma}\label{lem: pointwise convergence}
Given that Assumptions \ref{ass: waiting time}--\ref{ass: limiting process} hold, $t>0$, and $x>0$, then
\begin{align}\label{eq: pointwise convergence}
   \probability*{\max_{i\leq N}\sum_{j=1}^{\lfloor tc_N \rfloor}(A_{i,j}B_j-T_j)>xc_N}\LimitN 1-\exp\left(-\frac{t}{(x+\mu t)^{\beta}}\right).
\end{align}
\end{lemma}
\begin{proof}
The approach to prove this lemma is by analyzing upper and lower bounds of the probability given in \eqref{eq: pointwise convergence} and by proving that these bounds are sharp as $N\to\infty$. Thus, first we see that 
\begin{align}
&\probability*{\max_{i\leq N}\sum_{j=1}^{\lfloor tc_N \rfloor}(A_{i,j}B_j-T_j)>xc_N}\\
& \quad=\probability*{\max_{i\leq N}\sum_{j=1}^{\lfloor tc_N \rfloor}(A_{i,j}B_j-T_j)>xc_N\Bigg|  \max_{j\leq \lfloor tc_N \rfloor}B_j>(x+\mu t-\delta)\frac{c_N}{b_N}}\probability*{\max_{j\leq \lfloor tc_N \rfloor}B_j>(x+\mu t-\delta)\frac{c_N}{b_N}}\\
& \quad\quad+\probability*{\max_{i\leq N}\sum_{j=1}^{\lfloor tc_N \rfloor}(A_{i,j}B_j-T_j)>xc_N\Bigg|  \max_{j\leq \lfloor tc_N \rfloor}B_j\leq(x+\mu t-\delta)\frac{c_N}{b_N}}\probability*{\max_{j\leq \lfloor tc_N \rfloor}B_j\leq(x+\mu t-\delta)\frac{c_N}{b_N}}\\
& \quad\leq \probability*{\max_{j\leq \lfloor tc_N \rfloor}B_j>(x+\mu t-\delta)\frac{c_N}{b_N}}+\probability*{\max_{i\leq N}\sum_{j=1}^{\lfloor tc_N \rfloor}(A_{i,j}B_j-T_j)>xc_N\Bigg|  \max_{j\leq \lfloor tc_N \rfloor}B_j\leq(x+\mu t-\delta)\frac{c_N}{b_N}}\label{eq: upper bound pointwise convergence}.
\end{align}
The first term in \eqref{eq: upper bound pointwise convergence} yields
\begin{align*}
\probability*{\max_{j\leq \lfloor tc_N \rfloor}B_j>(x+\mu t-\delta)\frac{c_N}{b_N}}&\sim 1-\left(1-\frac{L((x+\mu t-\delta)\frac{c_N}{b_N})}{\left((x+\mu t-\delta)\frac{c_N}{b_N}\right)^{\beta}}\right)^{\lfloor tc_N \rfloor}\\
&\LimitN 1-\exp\left(-\frac{t}{(x+\mu t-\delta)^{\beta}}\right)\\
&\overset{\delta\downarrow 0}{\longrightarrow} 1-\exp\left(-\frac{t}{(x+\mu t)^{\beta}}\right).
\end{align*}
Hence, in order to prove that the upper bound of \eqref{eq: pointwise convergence} is asymptotically sharp, we are left with proving that the second term in \eqref{eq: upper bound pointwise convergence} vanishes as $N\to\infty$. We analyze this term as follows; first, we have that $\frac{(x+\mu t-\delta/2)}{(x+\mu t-\delta)}>1$ for $\delta$ small enough, thus we write $\frac{(x+\mu t-\delta/2)}{(x+\mu t-\delta)}=1+\epsilon$ with $\epsilon>0$. Second, we can bound the second term in \eqref{eq: upper bound pointwise convergence} as
\begin{align}
    &\probability*{\max_{i\leq N}\sum_{j=1}^{\lfloor tc_N \rfloor}(A_{i,j}B_j-T_j)>xc_N\Bigg|  \max_{j\leq \lfloor tc_N \rfloor} B_j\leq(x+\mu t-\delta)\frac{c_N}{b_N}}\\
    & \quad\leq\mathbb{P}\Bigg(\max_{i\leq N}\sum_{j=1}^{\lfloor tc_N \rfloor}(A_{i,j}\mathbbm{1}(A_{i,j}\leq (1+\epsilon)^{1-\alpha}b_N^{1-\alpha})B_j-T_j)+\max_{i\leq N}\sum_{j=1}^{\lfloor tc_N \rfloor}A_{i,j}\mathbbm{1}(A_{i,j}\geq (1+\epsilon)^{1-\alpha}b_N^{1-\alpha})B_j>xc_N\label{subeq: upper bound auxiliary process 1}\\
    &\quad\quad\quad\quad\Bigg|   \max_{j\leq \lfloor tc_N \rfloor}B_j\leq(x+\mu t-\delta)\frac{c_N}{b_N}\Bigg)\nonumber\\
       &  \quad\leq\probability*{\max_{i\leq N}\sum_{j=1}^{\lfloor tc_N \rfloor}A_{i,j}\mathbbm{1}(A_{i,j}<(1+\epsilon)^{1-\alpha}b_N^{1-\alpha})B_j>\left(\mathbb{E}[A_{i,j}B_j]t+\frac{\delta}{4}\right)c_N\Bigg|  \max_{j\leq \lfloor tc_N \rfloor}B_j\leq(x+\mu t-\delta)\frac{c_N}{b_N}}\label{subeq: term LLN}\\
    &\quad\quad+\probability*{\sum_{j=1}^{\lfloor tc_N \rfloor}-T_j>\left(-\mathbb{E}[T_j]t+\frac{\delta}{4}\right)c_N}\label{subeq: term arrival stream}\\
    &\quad\quad+\probability*{\max_{i\leq N}\sum_{j=1}^{\lfloor tc_N \rfloor}A_{i,j}\mathbbm{1}(A_{i,j}\geq (1+\epsilon)^{1-\alpha}b_N^{1-\alpha})B_j>\left(x+\mu t-\frac{\delta}{2}\right)c_N\Bigg|  \max_{j\leq \lfloor tc_N \rfloor}B_j\leq(x+\mu t-\delta)\frac{c_N}{b_N}}\label{subeq: extreme term}.
\end{align}
The upper bound in \eqref{subeq: upper bound auxiliary process 1} holds because for a sequence of numbers $(a_{i,j},i\geq 1,j\geq 1)$, we have that 
$$
\max_{i\leq N}\left(\sum_{j=1}^ka_{i,j}\right)\leq\max_{i\leq N}\sum_{j=1}^ka_{i,j}\mathbbm{1}(a_{i,j}\leq c)+\max_{i\leq N}\sum_{j=1}^ka_{i,j}\mathbbm{1}(a_{i,j}> c).$$ 
The upper bound from \eqref{subeq: upper bound auxiliary process 1} to \eqref{subeq: term LLN}, \eqref{subeq: term arrival stream}, and \eqref{subeq: extreme term} holds because of the union bound. The term in \eqref{subeq: term arrival stream} converges to 0 due to the law of large numbers. For the term in \eqref{subeq: term LLN}, we know by the union bound that
\begin{multline}\label{eq: lln part}
    \probability*{\max_{i\leq N}\sum_{j=1}^{\lfloor tc_N \rfloor}A_{i,j}\mathbbm{1}(A_{i,j}<(1+\epsilon)^{1-\alpha}b_N^{1-\alpha})B_j>\left(\mathbb{E}[A_{i,j}B_j]t+\frac{\delta}{4}\right)c_N\Bigg|  \max_{j\leq \lfloor tc_N \rfloor}B_j\leq(x+\mu t-\delta)\frac{c_N}{b_N}}\\
    \leq N \probability*{\sum_{j=1}^{\lfloor tc_N \rfloor}A_{i,j}\mathbbm{1}(A_{i,j}<(1+\epsilon)^{1-\alpha}b_N^{1-\alpha})B_j>\left(\mathbb{E}[A_{i,j}B_j]t+\frac{\delta}{4}\right)c_N\Bigg|  \max_{j\leq \lfloor tc_N \rfloor}B_j\leq(x+\mu t-\delta)\frac{c_N}{b_N}}.
\end{multline}
Now, since we have a probability of sums of almost surely bounded random variables, we can apply Bennett's inequality with the setting given in Lemma \ref{lem: bennett} and Corollary \ref{cor: bennett}. We see that $\mathbb{E}[A_{i,j}\mathbbm{1}(A_{i,j}<(1+\epsilon)^{1-\alpha}b_N^{1-\alpha})B_j\mid B_j\leq(x+\mu t-\delta)\frac{c_N}{b_N}]<\mathbb{E}[A_{i,j}B_j]$. Furthermore, we can choose $M$ as $M=(x+\mu t-\delta)(1+\epsilon)^{1-\alpha}b_N^{1-\alpha}\frac{c_N}{b_N}$, and $y$ as $y=\frac{\delta}{4}c_N$. Thus, 
$$
\frac{y}{M}=\frac{\delta}{4(x+\mu t-\delta)(1+\epsilon)^{1-\alpha}}b_N^{\alpha}=\frac{\delta}{4(x+\mu t-\delta)(1+\epsilon)^{1-\alpha}q}\log N.
$$
It is important to note here, that $y/M$ equals a constant times $\log N$. We now add a subscript $N$ to the variables $y,M$, and $\sigma_i$ to indicate sequences that change with $N$. Now, for $\beta>2$, $\limsup_{N\to\infty}\sigma_{i,N}^2<\infty$. Thus,
$$
\frac{y_NM_N}{\sum_{j=1}^{\lfloor tc_N \rfloor}\sigma_{i,N}^2}\LimitN \infty.
$$
Therefore, using the information that $y/M$ equals a constant times $\log N$ and by using Corollary \ref{cor: bennett}, we see that the exponent in Corollary \ref{cor: bennett} grows faster to infinity than $\log N$.
Thus, by applying Bennett's inequality, we get that the expression in \eqref{eq: lln part} converges to 0 as $N\to\infty$. When $1<\beta\leq 2$, $\sigma_{i,N}^2\LimitN\infty$, however, from Lemma \ref{lem: second moment alpha small} follows that $\sigma_{i,N}^2/(\frac{c_N}{b_N})\LimitN 0$. Therefore, $y_NM_N/\sum_{j=1}^{\lfloor tc_N \rfloor}\sigma_{i,N}^2\LimitN\infty$. Concluding, from Corollary \ref{cor: bennett} we again get that the expression in \eqref{eq: lln part} and therefore the expression in \eqref{subeq: term LLN} converge to 0.

Furthermore, for the term in \eqref{subeq: extreme term} we have that
\begin{multline*}
    \probability*{\max_{i\leq N}\sum_{j=1}^{\lfloor tc_N \rfloor}A_{i,j}\mathbbm{1}(A_{i,j}\geq (1+\epsilon)^{1-\alpha}b_N^{1-\alpha})B_j>\left(x+\mu t-\frac{\delta}{2}\right)c_N\Bigg|  \max_{j\leq \lfloor tc_N \rfloor}B_j\leq(x+\mu t-\delta)\frac{c_N}{b_N}}\\
    \leq \probability*{\max_{i\leq N}\sum_{j=1}^{\lfloor tc_N \rfloor}A_{i,j}\mathbbm{1}(A_{i,j}\geq (1+\epsilon)^{1-\alpha}b_N^{1-\alpha})>\frac{x+\mu t-\delta/2}{x+\mu t-\delta}b_N}.
\end{multline*}
We have $\frac{(x+\mu t-\delta/2)}{(x+\mu t-\delta)}=1+\epsilon$ with $\epsilon>0$, thus we can further simplify and bound this probability as follows:
\begin{align}
       &\probability*{\max_{i\leq N}\sum_{j=1}^{\lfloor tc_N \rfloor}A_{i,j}\mathbbm{1}(A_{i,j}\geq (1+\epsilon)^{1-\alpha}b_N^{1-\alpha})>(1+\epsilon)b_N}\nonumber\\
   &\quad  \leq   \probability*{\max_{i\leq N}\sum_{j=1}^{\lfloor tc_N \rfloor}A_{i,j}\mathbbm{1}(A_{i,j}\geq (1+\epsilon)^{1-\alpha}b_N^{1-\alpha})>(1+\epsilon)b_N\cap \max_{i\leq N}\max_{j\leq \lfloor tc_N \rfloor}A_{i,j}>(1+\epsilon)b_N}\label{subeq: brosset like ineq 1}\\
   &\quad\quad+  \probability*{\max_{i\leq N}\sum_{j=1}^{\lfloor tc_N \rfloor}A_{i,j}\mathbbm{1}(A_{i,j}\geq (1+\epsilon)^{1-\alpha}b_N^{1-\alpha})>(1+\epsilon)b_N\cap \max_{i\leq N}\max_{j\leq \lfloor tc_N \rfloor}A_{i,j}<(1+\epsilon)b_N}\label{subeq: brosset like ineq 2}.
\end{align}
Since 
$$
\frac{\max_{i\leq N}\max_{j\leq \lfloor tc_N \rfloor}A_{i,j} }{b_N}\LimitP 1,
$$
as $N\to\infty$, the term in \eqref{subeq: brosset like ineq 1} converges to 0 as $N\to\infty$, and we only need to focus on the term in \eqref{subeq: brosset like ineq 2}. Observe that by the union bound,
\begin{multline}
    \probability*{\max_{i\leq N}\sum_{j=1}^{\lfloor tc_N \rfloor}A_{i,j}\mathbbm{1}(A_{i,j}\geq (1+\epsilon)^{1-\alpha}b_N^{1-\alpha})>(1+\epsilon)b_N\cap \max_{i\leq N}\max_{j\leq \lfloor tc_N \rfloor}A_{i,j}<(1+\epsilon)b_N}\\
    \leq N  \probability*{\sum_{j=1}^{\lfloor tc_N \rfloor}A_{i,j}\mathbbm{1}((1+\epsilon)^{1-\alpha}b_N^{1-\alpha}\leq A_{i,j}\leq(1+\epsilon)b_N )>(1+\epsilon)b_N}.\label{subeq: brosset like ineq 3}
\end{multline}
Following the proof given in \cite[Lem.\ 8]{brosset2022large}, we assume without loss of generality that $q=1$ and choose $1/(1+\epsilon)^\alpha<q'<1$ and $q'<q''<1$. Now, we have by using Chernoff's bound, that for $\theta>0$,
\begin{multline*}
    N  \probability*{\sum_{j=1}^{\lfloor tc_N \rfloor}A_{i,j}\mathbbm{1}((1+\epsilon)^{1-\alpha}b_N^{1-\alpha}\leq A_{i,j}\leq(1+\epsilon)b_N )>(1+\epsilon)b_N}\\
    \leq N \left(1+\mathbb{E}\left[\exp\left(\theta A_{i,j}\right)\mathbbm{1}((1+\epsilon)^{1-\alpha}b_N^{1-\alpha}\leq A_{i,j}\leq(1+\epsilon)b_N)\right]\right)^{\lfloor tc_N \rfloor}\exp(-\theta(1+\epsilon)b_N).
\end{multline*}
Then, for $\theta=q'(1+\epsilon)^{\alpha-1}b_N^{\alpha-1}$, in \cite[Lem.\ 8]{brosset2022large} it is proven that for $N$ large enough
\begin{multline*}
\mathbb{E}\left[\exp\left(q'(1+\epsilon)^{\alpha-1}b_N^{\alpha-1}A_{i,j}\right)\mathbbm{1}((1+\epsilon)^{1-\alpha}b_N^{1-\alpha}\leq A_{i,j}\leq(1+\epsilon)b_N)\right]\\
\leq(1+q'(1+\epsilon)^{\alpha}b_N^{\alpha})\exp(q'-q''(1+\epsilon)^{\alpha(1-\alpha)}b_N^{\alpha(1-\alpha)}).
\end{multline*}
Now, by using the fact that $x>0$ we have the simple bound $1+x\leq \exp(x)$, and that $c_N=\tilde{L}(b_N)b_N^{\frac{\beta}{(\beta-1)}}$, it is easy to see that
$$
\left(1+(1+q'(1+\epsilon)^{\alpha}b_N^{\alpha})\exp(q'-q''(1+\epsilon)^{\alpha(1-\alpha)}b_N^{\alpha(1-\alpha)})\right)^{\lfloor tc_N \rfloor}\LimitN 1.
$$
Therefore, we know that Chernoff's bound with $\theta=q'(1+\epsilon)^{\alpha-1}b_N^{\alpha-1}$ applied to the expression in \eqref{subeq: brosset like ineq 3} satisfies
\begin{multline*}
    \limsup_{N\to\infty}N \left(1+\mathbb{E}\left[\exp\left(q'(1+\epsilon)^{\alpha-1}b_N^{\alpha-1}A_{i,j}\right)\mathbbm{1}((1+\epsilon)^{1-\alpha}b_N^{1-\alpha}\leq A_{i,j}\leq(1+\epsilon)b_N)\right]\right)^{\lfloor tc_N \rfloor}\\
    \cdot\exp(-q'(1+\epsilon)^{\alpha-1}b_N^{\alpha-1}(1+\epsilon)b_N)
    \leq \limsup_{N\to\infty} N\exp(-q'(1+\epsilon)^{\alpha}b_N^{\alpha}).
\end{multline*}
Since $q'>1/(1+\epsilon)^\alpha$, we have that $q'(1+\epsilon)^{\alpha}b_N^{\alpha}>\log N$ and therefore that $N\exp(- q'(1+\epsilon)^{\alpha}b_N^{\alpha})\LimitN 0$. Thus we can conclude that the expression in \eqref{subeq: brosset like ineq 3} converges to 0 as $N\to\infty$. From this, it follows that the term in \eqref{subeq: extreme term} converges to 0 as $N\to\infty$ as well, and we can conclude that the upper bound proposed in \eqref{eq: upper bound pointwise convergence} is asymptotically sharp. 

To prove a sharp lower bound for the probability in \eqref{eq: pointwise convergence}, observe that, because for a sequence $(a_{i,j},i\geq 1,j\geq 1)$ we have that $\max_{i\leq N}\sum_{j=1}^ka_{i,j}\geq\max_{i\leq N}\max_{j\leq k}a_{i,j}+\sum_{j=1,j\neq j^*}^k a_{i^*,j}$. 
\begin{multline}\label{eq: lower bound pointwise convergence}
    \liminf_{N\to\infty}\probability*{\max_{i\leq N}\sum_{j=1}^{\lfloor tc_N \rfloor}(A_{i,j}B_j-T_j)>xc_N}\\
    \geq \liminf_{N\to\infty}\probability*{\max_{i\leq N}A_{i,j^*(t)}\max_{j\leq \lfloor tc_N \rfloor}B_j-T_{j^*(t)}+\sum_{j=1,j\neq j^*(t)}^{\lfloor tc_N \rfloor}(A_{i^*(t),j}B_j-T_j)>xc_N},
\end{multline}
where $j^*(t)\in\argmax\{j:B_{j^*(t)}=\max_{j\leq \lfloor tc_N \rfloor}B_j\}$ and $i^*(t)\in\argmax\{i:A_{i,j^*(t)}=\max_{i\leq N}A_{i,j^*(t)}\}$. Because, $\max_{j\leq \lfloor tc_N \rfloor}B_j$ scales as $c_N/b_N$, we get that $\mathbb{E}[A]\max_{j\leq \lfloor tc_N \rfloor}B_j/c_N\LimitP 0$, as $N\to\infty$, and thus we have that $\sum_{j=1,j\neq j^*(t)}^{\lfloor tc_N \rfloor}(A_{i^*(t),j}B_j-T_j)/c_N\LimitP -\mu t$, cf.\ \cite[Thm.\ 1]{kesten1993convergence}. Furthermore,
$$
\probability*{\max_{i\leq N}\max_{j\leq \lfloor tc_N \rfloor}(A_{i,j}B_j)/c_N>x+\mu t}\LimitN 1-\exp(-t/(x+\mu t)^{\beta})
$$ 
and $T_{j^*(t)}/c_N\LimitP 0$ as $N\to\infty$. In conclusion, the lower bound in \eqref{eq: lower bound pointwise convergence} is sharp, as the limit is the same as the limit in \eqref{eq: pointwise convergence}.
\end{proof}

We have established pointwise convergence. In Lemma \ref{lem: conv in D sum}, we prove convergence in $D[0,T]$.
\begin{lemma}\label{lem: conv in D sum}
Given that Assumptions \ref{ass: waiting time}--\ref{ass: limiting process} hold, and $T>0$, then
\begin{align}\label{eq: process convergence sum}
\left(\frac{\max_{i\leq N}\sum_{j=1}^{\lfloor tc_N\rfloor}(A_{i,j}B_j-T_j)}{c_N},t\in[0,T]\right)\LimitD \left(X_t-\mu t,t\in[0,T]\right),
\end{align}
as $N\to\infty$.
\end{lemma} 
This lemma follows from the two results stated in Lemma \ref{lem: conv in D max B} and \ref{lem: sup max B and process}.
\begin{lemma}\label{lem: conv in D max B}
Given that Assumptions \ref{ass: waiting time}--\ref{ass: limiting process} hold, and $T>0$, then
\begin{align}\label{eq: process convergence max B}
    \left(\frac{\max_{j\leq \lfloor tc_N\rfloor}B_j}{\frac{c_N}{b_N}},t\in[0,T]\right)\LimitD\left(X_t,t\in[0,T]\right),
\end{align}
as $N\to\infty$.
\end{lemma}
\begin{lemma}\label{lem: sup max B and process}
Given that Assumptions \ref{ass: waiting time}--\ref{ass: limiting process} hold, and $T>0$, then we have that for all $\epsilon>0$
\begin{align}\label{eq: process convergence diff}
\probability*{\sup_{t\in[0,T]}\bigg|\frac{\max_{i\leq N}\sum_{j=1}^{\lfloor tc_N\rfloor}(A_{i,j}B_j-T_j)}{c_N}-\left(\frac{\max_{j\leq \lfloor tc_N\rfloor}B_j}{\frac{c_N}{b_N}}-\mu t\right)\bigg|>\epsilon}\LimitN 0.
\end{align}
\end{lemma}
 Using the triangle inequality, we get that \eqref{eq: process convergence sum} follows from \eqref{eq: process convergence max B} and \eqref{eq: process convergence diff}.
\begin{proof}[Proof of Lemma \ref{lem: conv in D max B}]
In this proof, we use Lemma \ref{lem: billingsley conv in D}, thus we need to prove the three conditions stated in Lemma \ref{lem: billingsley conv in D}.
First of all, we need to prove that 
$$
\left(\frac{\max_{j\leq \lfloor t_1c_N\rfloor}B_j}{\frac{c_N}{b_N}},\ldots,\frac{\max_{j\leq \lfloor t_mc_N\rfloor}B_j}{\frac{c_N}{b_N}}\right)\LimitD\left(X_{t_1},\ldots,X_{t_m}\right),
$$
as $N\to\infty$. Let us assume that $m=2$ and $t_2>t_1$. If $x_2\leq x_1$, because $\max_{j\leq k}B_j$ is increasing in $k$, we have that 
\begin{multline*}
    \probability*{\frac{\max_{j\leq \lfloor t_1c_N\rfloor}B_j}{\frac{c_N}{b_N}}\leq x_1 \cap\frac{\max_{j\leq \lfloor t_2c_N\rfloor}B_j}{\frac{c_N}{b_N}}\leq x_2}= \probability*{\frac{\max_{j\leq \lfloor t_2c_N\rfloor}B_j}{\frac{c_N}{b_N}}\leq x_2}\\
    \LimitN \probability*{X_{t_2}\leq x_2}=\probability*{X_{t_1}\leq x_1\cap X_{t_2}\leq x_2}.
\end{multline*}
When $x_2>x_1$, we have that 
\begin{align*}
    &\probability*{\frac{\max_{j\leq \lfloor t_1c_N\rfloor}B_j}{\frac{c_N}{b_N}}\leq x_1 \cap\frac{\max_{j\leq \lfloor t_2c_N\rfloor}B_j}{\frac{c_N}{b_N}}\leq x_2}\\
    & \quad= \probability*{\frac{\max_{j\leq \lfloor t_2c_N\rfloor}B_j}{\frac{c_N}{b_N}}\leq x_2\Bigg|\frac{\max_{j\leq \lfloor t_1c_N\rfloor}B_j}{\frac{c_N}{b_N}}\leq x_1}\probability*{\frac{\max_{j\leq \lfloor t_1c_N\rfloor}B_j}{\frac{c_N}{b_N}}\leq x_1}\\
    & \quad=\probability*{\frac{\max_{j\leq \lfloor t_2c_N\rfloor-\lfloor t_1c_N\rfloor}B_j}{\frac{c_N}{b_N}}\leq x_2}\probability*{\frac{\max_{j\leq \lfloor t_1c_N\rfloor}B_j}{\frac{c_N}{b_N}}\leq x_1}\\
    & \quad\LimitN \probability{X_{t_2-t_1}\leq x_2}\probability{X_{t_1}\leq x_1}=\probability{X_{t_1}\leq x_1\cap X_{t_2}\leq x_2}.
\end{align*}
For $m>2$ but finite, analogous derivations hold. Second, we need to prove that 
$$
X_T-X_{T-\delta}\overset{\mathbb{P}}{\longrightarrow} 0,
$$
as $\delta\downarrow 0$. We can write $X_T=\max({X_{T-\delta},\hat{X}_{\delta}})$. Therefore, $X_T-X_{T-\delta}\leq \hat{X}_{\delta}$. Let $\epsilon>0$, then 
$$
\probability{X_T-X_{T-\delta}>\epsilon}\leq \probability{\hat{X}_{\delta}>\epsilon}=1-\exp\left(-\frac{\delta}{\epsilon^{\beta}}\right)\overset{\delta\downarrow 0}{\longrightarrow}0.
$$
Finally, we show that the process $ \left(\frac{\max_{j\leq \lfloor tc_N\rfloor}B_j}{(\frac{c_N}{b_N})},t\in[0,T]\right)$ satisfies the third condition in Lemma \ref{lem: billingsley conv in D}. The random variable $\max_{j\leq k}B_j$ is increasing with $k$. Furthermore, the minimum of two numbers is bounded from above by the average. Also, because for $k>l$, $\max_{j\leq k}B_j-\max_{j\leq l}B_j=\max(\max_{j\leq l}B_j,\max_{l+1\leq j \leq k}B_j)-\max_{j\leq l}B_j$, we can bound
$$\frac{\max_{j\leq \lfloor sc_N\rfloor}B_j-\max_{j\leq \lfloor rc_N\rfloor}B_j}{\frac{c_N}{b_N}}\leq_{st.} \frac{\max_{j\leq \lfloor sc_N\rfloor-\lfloor rc_N\rfloor}\hat{B}_j}{\frac{c_N}{b_N}},
$$
where $\hat{B}$ is an independent copy of $B$. Therefore, we have that 
\begin{align*}
 &\sup_{s\in[r,t]} \min\bigg|\frac{\max_{j\leq \lfloor sc_N\rfloor}B_j-\max_{j\leq \lfloor rc_N\rfloor}B_j}{\frac{c_N}{b_N}},\frac{\max_{j\leq \lfloor tc_N\rfloor}B_j-\max_{j\leq \lfloor sc_N\rfloor}B_j}{\frac{c_N}{b_N}}\bigg|\\
 &\quad=\sup_{s\in[r,t]} \min\left(\frac{\max_{j\leq \lfloor sc_N\rfloor}B_j-\max_{j\leq \lfloor rc_N\rfloor}B_j}{\frac{c_N}{b_N}},\frac{\max_{j\leq \lfloor tc_N\rfloor}B_j-\max_{j\leq \lfloor sc_N\rfloor}B_j}{\frac{c_N}{b_N}}\right)\\
  &\quad \leq \frac{\max_{j\leq \lfloor tc_N\rfloor}B_j-\max_{j\leq \lfloor rc_N\rfloor}B_j}{2\frac{c_N}{b_N}}\\
  &\quad\leq_{st.} \frac{\max_{j\leq \lfloor tc_N\rfloor-\lfloor rc_N\rfloor}\hat{B}_j}{2\frac{c_N}{b_N}}.
\end{align*}
Thus, using the expression in the third condition of Lemma \ref{lem: billingsley conv in D}, we obtain that
\begin{align*}
    &\probability*{ \sup_{s\in[r,t], t-r<\delta}\min\bigg|\frac{\max_{j\leq \lfloor sc_N\rfloor}B_j-\max_{j\leq \lfloor rc_N\rfloor}B_j}{\frac{c_N}{b_N}},\frac{\max_{j\leq \lfloor tc_N\rfloor}B_j-\max_{j\leq \lfloor sc_N\rfloor}B_j}{\frac{c_N}{b_N}}\bigg|>\epsilon}\\
    &\quad\leq\probability*{\frac{\max_{j\leq \lfloor \delta c_N\rfloor}\hat{B}_j}{\frac{c_N}{b_N}}>2\epsilon}\\
    &\quad\leq \lfloor \delta c_N\rfloor\frac{L(2\epsilon \frac{c_N}{b_N})}{(2\epsilon \frac{c_N}{b_N})^{\beta}}.
\end{align*}
We have that $c_N\frac{L(2\epsilon\frac{c_N}{b_N})}{(c_N/b_N)^{\beta}}\LimitN 1$ because $c_N\sim\frac{(c_N/b_N)^{\beta}}{L(c_N/b_N)}$ as $N\to\infty$ and $L$ is a slowly varying function, so we choose $N_0>1$ such that $\lfloor\delta c_N\rfloor\frac{L(2\epsilon \frac{c_N}{b_N})}{(2\epsilon \frac{c_N}{b_N})^{\beta}}<(1+\epsilon)\frac{\delta}{(2\epsilon)^{\beta}}$ for all $N>N_0$. Now, choose $0<\delta<\frac{\eta(2\epsilon)^{\beta}}{(1+\epsilon)}$ and we get that for $N>N_0$
\begin{align*}
    &\probability*{ \sup_{s\in[r,t],t-r<\delta}\min\bigg|\frac{\max_{j\leq \lfloor sc_N\rfloor}B_j-\max_{j\leq \lfloor rc_N\rfloor}B_j}{\frac{c_N}{b_N}},\frac{\max_{j\leq \lfloor tc_N\rfloor}B_j-\max_{j\leq \lfloor sc_N\rfloor}B_j}{\frac{c_N}{b_N}}\bigg|>\epsilon}<\eta.
    \end{align*}
Hence, the process $ \left(\frac{\max_{j\leq \lfloor tc_N\rfloor}B_j}{(\frac{c_N}{b_N})},t\in[0,T]\right)$ also satisfies the third condition in Lemma \ref{lem: billingsley conv in D} and the result follows.
\end{proof}
We have proven process convergence of $\left(\max_{j\leq \lfloor tc_N\rfloor}B_j/(\frac{c_N}{b_N}),t\in[0,T]\right)$ to $(X_t,t\in[0,T])$, now in order to prove Lemma \ref{lem: conv in D sum} we are left by proving that the convergence result in \eqref{eq: process convergence diff} holds. We do this in Lemma \ref{lem: sup max B and process}.

\begin{proof}[Proof of Lemma \ref{lem: sup max B and process}]
The random variable in \eqref{eq: process convergence diff} has the form of a supremum of the absolute value of a stochastic process. We know that $|X|=\max(X,-X)$. Then, by applying the union bound we get that $\mathbb{P}(|X|>x)\leq \probability{X>x}+\probability{-X>x}$. Thus, to prove the convergence result in \eqref{eq: process convergence diff} we can remove the absolute value and need to prove that the probability of a supremum of a stochastic process converges to 0 as $N\to\infty$, cf.\ \eqref{eq: process convergence auxiliary process}, and we need to prove that the probability of the supremum of the mirrored process converges to 0 as $N\to\infty$, cf.\ \eqref{eq: process convergence mirrored auxiliary process}. We first prove that 
\begin{align}\label{eq: process convergence auxiliary process}
\probability*{\sup_{t\in[0,T]}\left(-\frac{\max_{i\leq N}\sum_{j=1}^{\lfloor tc_N\rfloor}(A_{i,j}B_j-T_j)}{c_N}+\left(\frac{\max_{j\leq \lfloor tc_N\rfloor}B_j}{\frac{c_N}{b_N}}-\mu t\right)\right)>\epsilon}
\end{align}
converges to 0 as $N\to\infty$.
We have, by using $i^*(t)$ and $j^*(t)$ as defined in Lemma \ref{lem: pointwise convergence} that 
\begin{align}
&\probability*{\sup_{t\in[0,T]}\left(-\frac{\max_{i\leq N}\sum_{j=1}^{\lfloor tc_N\rfloor}(A_{i,j}B_j-T_j)}{c_N}+\left(\frac{\max_{j\leq \lfloor tc_N\rfloor}B_j}{\frac{c_N}{b_N}}-\mu t\right)\right)>\epsilon}\nonumber\\
&\quad\leq\mathbb{P}\Bigg(\sup_{t\in[0,T]}\Bigg(-\mu t-\frac{\sum_{j=1,j\neq j^*(t)}^{\lfloor tc_N \rfloor}(A_{i^*(t),j}B_j-T_j)-T_{j^*(t)}}{c_N}+\frac{\max_{j\leq \lfloor tc_N\rfloor}B_j}{\frac{c_N}{b_N}}\nonumber\\
&\quad\quad-\frac{\max_{i\leq N}A_{i,j^*(t)}\max_{j\leq \lfloor tc_N \rfloor}B_j}{c_N}\Bigg)>\epsilon\Bigg)\nonumber\\
&\quad\leq \probability*{\sup_{t\in[0,T]}\left(-\mu t-\frac{\sum_{j=1,j\neq j^*(t)}^{\lfloor tc_N \rfloor}(A_{i^*(t),j}B_j-T_j)-T_{j^*(t)}}{c_N}\right)>\frac{\epsilon}{2}}\label{subeq: law of large numbers sup}\\
&\quad\quad+\probability*{\sup_{t\in[0,T]}\left(\frac{\max_{j\leq \lfloor tc_N\rfloor}B_j}{\frac{c_N}{b_N}}-\frac{\max_{i\leq N}A_{i,j^*(t)}\max_{j\leq \lfloor tc_N \rfloor}B_j}{c_N}\right)>\frac{\epsilon}{2}}.\label{subeq: convergence maximum sup}
\end{align}
For the term in \eqref{subeq: law of large numbers sup}, we use the union bound to obtain that 
\begin{align}
    &\probability*{\sup_{t\in[0,T]}\left(-\mu t-\frac{\sum_{j=1,j\neq j^*(t)}^{\lfloor tc_N \rfloor}(A_{i^*(t),j}B_j-T_j)-T_{j^*(t)}}{c_N}\right)>\frac{\epsilon}{2}}\nonumber\\
    &\quad\leq \probability*{\sup_{t\in[0,\epsilon/(4\mathbb{E}[T_j])]}\left(-\mu t-\frac{\sum_{j=1,j\neq j^*(t)}^{\lfloor tc_N \rfloor}(A_{i^*(t),j}B_j-T_j)-T_{j^*(t)}}{c_N}\right)>\frac{\epsilon}{2}}\label{subeq: law of large numbers sup 1}\\
    &\quad\quad +\probability*{\sup_{t\in[\epsilon/(4\mathbb{E}[T_j]),T]}\left(-\mu t-\frac{\sum_{j=1,j\neq j^*(t)}^{\lfloor tc_N \rfloor}(A_{i^*(t),j}B_j-T_j)-T_{j^*(t)}}{c_N}\right)>\frac{\epsilon}{2}}.\label{subeq: law of large numbers sup 2}
\end{align}
Because all random variables $A_{i,j}$, $B_j$, and $T_j$ are positive, it is easy to see that the term in \eqref{subeq: law of large numbers sup 1} can be upper bounded by
$$
\probability*{\sup_{t\in[0,\epsilon/(4\mathbb{E}[T_j])]}\left(\frac{\sum_{j=1}^{\lfloor tc_N \rfloor}T_j}{c_N}\right)>\frac{\epsilon}{2}}=\probability*{\frac{\sum_{j=1}^{\lfloor \epsilon/(4\mathbb{E}[T_j])c_N \rfloor}T_j}{c_N}>\frac{\epsilon}{2}}\LimitN 0,
$$
as we can conclude from the law of large numbers that $\frac{\sum_{j=1}^{\lfloor \epsilon/(4\mathbb{E}[T_j])c_N \rfloor}T_j}{c_N}\LimitP \epsilon/4$ as $N\to\infty$.
For the term in \eqref{subeq: law of large numbers sup 2} we have for $0<\delta<1$, since all random variables are positive, that 
\begin{align}
&\probability*{\sup_{t\in[\epsilon/(4\mathbb{E}[T_j]),T]}\left(-\mu t-\frac{\sum_{j=1,j\neq j^*(t)}^{\lfloor tc_N \rfloor}(A_{i^*(t),j}B_j-T_j)-T_{j^*(t)}}{c_N}\right)>\frac{\epsilon}{2}}\nonumber\\
    &\quad\leq \sup_{t\in[\epsilon/(4\mathbb{E}[T_j]),T]}\frac{1}{\delta}\probability*{\sup_{s\in[t,t+\delta]}\left(-\mu s-\frac{\sum_{j=1,j\neq j^*(s)}^{\lfloor sc_N\rfloor}(A_{i^*(s),j}B_j-T_j)-T_{j^*(s)}}{c_N}\right)>\frac{\epsilon}{2}}\nonumber\\
    &\quad\leq \sup_{t\in[\epsilon/(4\mathbb{E}[T_j]),T]}\frac{1}{\delta}\probability*{\left(-\mu t-\frac{\inf_{s\in[t,t+\delta]}\sum_{j=1,j\neq j^*(t)}^{\lfloor tc_N\rfloor}A_{i^*(s),j}B_j}{c_N}+\frac{\sum_{j=1}^{\lfloor (t+\delta)c_N\rfloor}T_j}{c_N}\right)>\frac{\epsilon}{2}}.\label{eq: convergence supremum law of large numbers D[0,T]}
\end{align}
To bound the term in \eqref{eq: convergence supremum law of large numbers D[0,T]}, we use the result from \cite[Eq.\ (6)]{godreche2017record} that the expected number of new extremes of the process $(\max_{j\leq \lfloor sc_N\rfloor}B_j,s\geq 0)$ in the interval $[t,t+\delta]$ equals $\sum_{j=\lfloor tc_N\rfloor}^{\lfloor (t+\delta)c_N\rfloor}1/j\LimitN \log((t+\delta)/t)$. Therefore, we can conclude that the number of different instances of $i^*(s)$ when $s\in[t,t+\delta]$ is asymptotically finite, with probability converging to 1, and therefore, we can use the union bound to bound 
\begin{align*}
    &\sup_{t\in[\epsilon/(4\mathbb{E}[T_j]),T]}\frac{1}{\delta}\probability*{\left(-\mu t-\frac{\inf_{s\in[t,t+\delta]}\sum_{j=1,j\neq j^*(t)}^{\lfloor tc_N\rfloor}A_{i^*(s),j}B_j}{c_N}+\frac{\sum_{j=1}^{\lfloor (t+\delta)c_N\rfloor}T_j}{c_N}\right)>\frac{\epsilon}{2}}\\
    &\quad\leq \sup_{t\in[\epsilon/(4\mathbb{E}[T_j]),T]}\sum_{k=0}^{\infty}(k+1)\probability*{\# \text{ new extremes of } \left(\max_{j\leq \lfloor sc_N\rfloor}B_j,s\geq 0\right) \text{ in } [t,t+\delta]=k}\\
    &\quad\quad\cdot\frac{1}{\delta}\probability*{\left(-\mu t-\frac{\sum_{j=1,j\neq j^*(t)}^{\lfloor tc_N\rfloor}A_{i,j}B_j}{c_N}+\frac{\sum_{j=1}^{\lfloor (t+\delta)c_N\rfloor}T_j}{c_N}\right)>\frac{\epsilon}{2}}\\
    &\quad=\sup_{t\in[\epsilon/(4\mathbb{E}[T_j]),T]}\left(1+\sum_{j=\lfloor tc_N\rfloor}^{\lfloor (t+\delta)c_N\rfloor}\frac{1}{j}\right)\frac{1}{\delta}\probability*{\left(-\mu t-\frac{\sum_{j=1,j\neq j^*(t)}^{\lfloor tc_N\rfloor}A_{i,j}B_j}{c_N}+\frac{\sum_{j=1}^{\lfloor (t+\delta)c_N\rfloor}T_j}{c_N}\right)>\frac{\epsilon}{2}}.
\end{align*}
This last expression converges to 0 as $N\to\infty$, when $\delta>0$ is small enough; that is $0<\delta<\epsilon/(2\mathbb{E}[T_j])$, because by the law of large numbers we obtain that $-\mu t-\frac{\sum_{j=1,j\neq j^*(t)}^{\lfloor tc_N\rfloor}A_{i,j}B_j}{c_N}+\frac{\sum_{j=1}^{\lfloor (t+\delta)c_N\rfloor}T_j}{c_N}\LimitP -\mu t-\mathbb{E}[A_{i,j}B_j]t+(t+\delta)\mathbb{E}[T_j]=\mathbb{E}[T_j]\delta$ as $N\to\infty$. Thus, we can conclude that the expression in \eqref{eq: convergence supremum law of large numbers D[0,T]}, and therefore also in \eqref{subeq: law of large numbers sup} converge to 0, as $N\to\infty$. For the term in \eqref{subeq: convergence maximum sup}, we have that
\begin{multline}\label{eq: upper bound max B max A max B}
   \probability*{\sup_{t\in[0,T]}\left(\frac{\max_{j\leq \lfloor tc_N\rfloor}B_j}{\frac{c_N}{b_N}}-\frac{\max_{i\leq N}A_{i,j^*(t)}\max_{j\leq \lfloor tc_N \rfloor}B_j}{c_N}\right)>\frac{\epsilon}{2}}\\
    \leq \probability*{\sup_{t\in[0,T]}\left(1-\frac{\max_{i\leq N}A_{i,\lfloor tc_N\rfloor}}{b_N}\right)\frac{\max_{j\leq \lfloor T c_N\rfloor}B_j}{\frac{c_N}{b_N}}>\frac{\epsilon}{2}}.
\end{multline}
This tail probability converges to 0 as $N\to\infty$, since we know that $\frac{\max_{j\leq \lfloor T c_N\rfloor}B_j}{(\frac{c_N}{b_N})}$ converges in distribution to a Fr\'echet random variable as $N\to\infty$, and $\sup_{t\in[0,T]}\left(1-\frac{\max_{i\leq N}A_{i,\lfloor tc_N\rfloor}}{b_N}\right)\LimitP 0,$ as $N\to\infty$. To see this, we first bound
\begin{align*}
   \probability*{\sup_{t\in[0,T]}\left(1-\frac{\max_{i\leq N}A_{i,\lfloor tc_N\rfloor}}{b_N}\right)>\frac{\epsilon}{2}}=&\probability*{\inf_{t\in[0,T]}\frac{\max_{i\leq N}A_{i,\lfloor tc_N\rfloor}}{b_N}<1-\frac{\epsilon}{2}}\nonumber\\
   \leq& \lfloor Tc_N\rfloor\probability*{\frac{\max_{i\leq N}A_{i,1}}{b_N}<1-\frac{\epsilon}{2}}.
\end{align*}
Now, we have that $\probability*{\frac{\max_{i\leq N}A_{i,1}}{b_N}<1-\frac{\epsilon}{2}}\leq\exp\left(-N\probability*{\frac{A_{i,1}}{b_N}>1-\frac{\epsilon}{2}}\right)$, cf.\ the proof of \cite[Thm.\ 5.4.4, p.\ 192]{de2007extreme}; thus
\begin{align*}
    \lfloor Tc_N\rfloor\probability*{\frac{\max_{i\leq N}A_{i,1}}{b_N}<1-\frac{\epsilon}{2}}\leq&\lfloor Tc_N\rfloor\exp\left(-N\probability*{\frac{A_{i,1}}{b_N}>1-\frac{\epsilon}{2}}\right)\\
    =&\lfloor Tc_N\rfloor\exp\left(-N\exp(-(1+o(1))(1-\epsilon/2)^{\alpha}\log N)\right)\\
    =&\lfloor Tc_N\rfloor\exp\left(-N^{1-(1+o(1))(1-\epsilon/2)^{\alpha}}\right)\\
    \LimitN& 0.
\end{align*}
Hence, the upper bound in \eqref{eq: upper bound max B max A max B} converges to 0 as $N\to\infty$. These results together give that the tail probability in \eqref{eq: process convergence auxiliary process} converges to 0 as $N\to\infty$.

To prove the convergence result in \eqref{eq: process convergence diff}, we are left with proving that the probability
\begin{align}\label{eq: process convergence mirrored auxiliary process}
    \probability*{\sup_{t\in[0,T]}\left(\frac{\max_{i\leq N}\sum_{j=1}^{\lfloor tc_N\rfloor}(A_{i,j}B_j-T_j)}{c_N}-\left(\frac{\max_{j\leq \lfloor tc_N\rfloor}B_j}{\frac{c_N}{b_N}}-\mu t\right)\right)>\epsilon}
\end{align}
converges to 0 as $N\to\infty$. In order to do so, we use the upper bound
\begin{multline*}
\probability*{\sup_{t\in[0,T]}\left(\frac{\max_{i\leq N}\sum_{j=1}^{\lfloor tc_N\rfloor}(A_{i,j}B_j-T_j)}{c_N}-\left(\frac{\max_{j\leq \lfloor tc_N\rfloor}B_j}{\frac{c_N}{b_N}}-\mu t\right)\right)>\epsilon}\\
\leq \sup_{t\in[0,T]}\frac{1}{\delta}\probability*{\sup_{s\in[t,t+\delta]}\left(\frac{\max_{i\leq N}\sum_{j=1}^{\lfloor sc_N\rfloor}(A_{i,j}B_j-T_j)}{c_N}-\left(\frac{\max_{j\leq \lfloor sc_N\rfloor}B_j}{\frac{c_N}{b_N}}-\mu s\right)\right)>\epsilon}
\end{multline*}
with $0<\delta<1$.
Now, we bound this further as follows;
\begin{align}\label{eq: upper bound delta small}
   &\frac{1}{\delta}\probability*{\sup_{s\in[t,t+\delta]}\left(\frac{\max_{i\leq N}\sum_{j=1}^{\lfloor sc_N\rfloor}(A_{i,j}B_j-T_j)}{c_N}-\left(\frac{\max_{j\leq \lfloor sc_N\rfloor}B_j}{\frac{c_N}{b_N}}-\mu s\right)\right)>\epsilon}\\
   & \quad\leq \frac{1}{\delta}\probability*{\sup_{s\in[t,t+\delta]}\left(\frac{\max_{i\leq N}\sum_{j=1}^{\lfloor sc_N\rfloor}(A_{i,j}\mathbbm{1}(A_{i,j}<b_N^{1-\alpha})B_j-T_j)}{c_N}+\mu s\right)>\frac{\epsilon}{2}}\nonumber\\
   &\quad \quad+\frac{1}{\delta}\probability*{\sup_{s\in[t,t+\delta]}\left(\frac{\max_{i\leq N}\sum_{j=1}^{\lfloor sc_N\rfloor}A_{i,j}\mathbbm{1}(A_{i,j}>b_N^{1-\alpha})B_j}{c_N}-\frac{\max_{j\leq \lfloor sc_N\rfloor}B_j}{\frac{c_N}{b_N}}\right)>\frac{\epsilon}{2}}\nonumber.
\end{align}
The first term vanishes asymptotically, by taking $\delta$ small enough compared to $\epsilon$, by using a similar argument as for bounding \eqref{eq: convergence supremum law of large numbers D[0,T]} and by using the same argument as in Lemma \ref{lem: pointwise convergence}. For the second term, we see that from Lemma \ref{lem: pointwise convergence} that 
$$
\left(\frac{\max_{i\leq N}\sum_{j=1}^{\lfloor (t+\delta)c_N\rfloor}A_{i,j}\mathbbm{1}(A_{i,j}>b_N^{1-\alpha})}{b_N}-1\right)\LimitP 0,
$$
as $N\to\infty$. We also know that $\frac{\max_{j\leq \lfloor (t+\delta)c_N\rfloor}B_j}{(\frac{c_N}{b_N})}$ converges in distribution as $N\to\infty$. Now, we can conclude that
\begin{multline*}
    \frac{1}{\delta}\probability*{\sup_{s\in[t,t+\delta]}\left(\frac{\max_{i\leq N}\sum_{j=1}^{\lfloor sc_N\rfloor}A_{i,j}\mathbbm{1}(A_{i,j}>b_N^{1-\alpha})B_j}{c_N}-\frac{\max_{j\leq \lfloor sc_N\rfloor}B_j}{\frac{c_N}{b_N}}\right)>\frac{\epsilon}{2}}  \\
    \leq \frac{1}{\delta}\probability*{\left(\frac{\max_{i\leq N}\sum_{j=1}^{\lfloor (t+\delta)c_N\rfloor}A_{i,j}\mathbbm{1}(A_{i,j}>b_N^{1-\alpha})}{b_N}-1\right)\frac{\max_{j\leq \lfloor (t+\delta)c_N\rfloor}B_j}{\frac{c_N}{b_N}}>\frac{\epsilon}{2}}\LimitN 0.
\end{multline*}
Now we have established that the probabilities in \eqref{eq: process convergence auxiliary process} and \eqref{eq: process convergence mirrored auxiliary process} converge to 0 as $N\to\infty$, the result follows.
\end{proof}
From the results in Lemmas \ref{lem: conv in D sum} and \ref{lem: conv in D max B}, we can conclude that the convergence result in \eqref{eq: process convergence sum} holds. Furthermore, by applying the continuous mapping theorem, Theorem \ref{thm: auxiliary result waiting time heavy tail} follows. 
\section{Process convergence of the maximum waiting time in $D[0,T]$}
At this point, we have proven the convergence result of an auxiliary process whose marginals are the same as the marginals of the maximum waiting time. Now, we can extend these results to prove convergence of the maximum waiting time $(\frac{\max_{i\leq N}W_i(tc_N)}{c_N},t\in[0,T])$ to the process $(\sup_{s\in[0,t]}(X_{(s,t)}-\mu(t-s)),t\in[0,T])$ as $N\to\infty$. We first show in Lemma \ref{lem: sup max B and process2} that the maximum waiting time can be approximated by an auxiliary process, as we did in Lemma \ref{lem: sup max B and process}, and then we prove the main result described in Theorem \ref{thm: main result waiting time heavy tail}.
\begin{lemma}\label{lem: sup max B and process2}
Given that Assumptions \ref{ass: waiting time}--\ref{ass: limiting process} hold, and $T>0$, then we have that for all $\epsilon>0$
\begin{align}\label{eq: process convergence diff 2}
\probability*{\sup_{t\in[0,T]}\bigg|\sup_{s\in[0,t]}\frac{\max_{i\leq N}\sum_{j=\lfloor sc_N\rfloor}^{\lfloor tc_N\rfloor}(A_{i,j}B_j-T_j)}{c_N}-\sup_{s\in[0,t]}\left(\frac{\max_{\lfloor sc_N\rfloor\leq j\leq \lfloor tc_N\rfloor}B_j}{\frac{c_N}{b_N}}-\mu (t-s)\right)\bigg|>\epsilon}\LimitN 0.
\end{align}
\end{lemma}
\begin{proof}
As in Lemma \ref{lem: sup max B and process}, we first use that $|X|=\max(X,-X)$. Then, by applying the union bound we get that $\mathbb{P}(|X|>x)\leq \probability{X>x}+\probability{-X>x}$. Now, we have that 
\begin{multline*}
    \sup_{t\in[0,T]}\left(\sup_{s\in[0,t]}\frac{\max_{i\leq N}\sum_{j=\lfloor sc_N\rfloor}^{\lfloor tc_N\rfloor}(A_{i,j}B_j-T_j)}{c_N}-\sup_{s\in[0,t]}\left(\frac{\max_{\lfloor sc_N\rfloor\leq j\leq \lfloor tc_N\rfloor}B_j}{\frac{c_N}{b_N}}-\mu (t-s)\right)\right)\\
    \leq \sup_{t\in[0,T]}\sup_{s\in[0,t]}\left(\frac{\max_{i\leq N}\sum_{j=\lfloor sc_N\rfloor}^{\lfloor tc_N\rfloor}(A_{i,j}B_j-T_j)}{c_N}-\left(\frac{\max_{\lfloor sc_N\rfloor\leq j\leq \lfloor tc_N\rfloor}B_j}{\frac{c_N}{b_N}}-\mu (t-s)\right)\right).
\end{multline*}
Similarly,
\begin{multline*}
      \sup_{t\in[0,T]}\left(\sup_{s\in[0,t]}\left(\frac{\max_{\lfloor sc_N\rfloor\leq j\leq \lfloor tc_N\rfloor}B_j}{\frac{c_N}{b_N}}-\mu (t-s)\right)-\sup_{s\in[0,t]}\frac{\max_{i\leq N}\sum_{j=\lfloor sc_N\rfloor}^{\lfloor tc_N\rfloor}(A_{i,j}B_j-T_j)}{c_N}\right)\\
    \leq \sup_{t\in[0,T]}\sup_{s\in[0,t]}\left(\left(\frac{\max_{\lfloor sc_N\rfloor\leq j\leq \lfloor tc_N\rfloor}B_j}{\frac{c_N}{b_N}}-\mu (t-s)\right)-\frac{\max_{i\leq N}\sum_{j=\lfloor sc_N\rfloor}^{\lfloor tc_N\rfloor}(A_{i,j}B_j-T_j)}{c_N}\right).
\end{multline*}
Therefore,
\begin{multline*}
    \probability*{\sup_{t\in[0,T]}\bigg|\sup_{s\in[0,t]}\frac{\max_{i\leq N}\sum_{j=\lfloor sc_N\rfloor}^{\lfloor tc_N\rfloor}(A_{i,j}B_j-T_j)}{c_N}-\sup_{s\in[0,t]}\left(\frac{\max_{\lfloor sc_N\rfloor\leq j\leq \lfloor tc_N\rfloor}B_j}{\frac{c_N}{b_N}}-\mu (t-s)\right)\bigg|>\epsilon}\\
    \leq 2\probability*{\sup_{t\in[0,T]}\sup_{s\in[0,t]}\bigg|\frac{\max_{i\leq N}\sum_{j=\lfloor sc_N\rfloor}^{\lfloor tc_N\rfloor}(A_{i,j}B_j-T_j)}{c_N}-\left(\frac{\max_{\lfloor sc_N\rfloor\leq j\leq \lfloor tc_N\rfloor}B_j}{\frac{c_N}{b_N}}-\mu (t-s)\right)\bigg|>\epsilon}.
\end{multline*}
Now, we use the same approach as in Lemma \ref{lem: sup max B and process}, with the somewhat different upper bound:
\begin{multline*}
    \probability*{\sup_{t\in[0,T]}\sup_{s\in[0,t]}\left(\frac{\max_{i\leq N}\sum_{j=\lfloor sc_N\rfloor}^{\lfloor tc_N\rfloor}(A_{i,j}B_j-T_j)}{c_N}-\left(\frac{\max_{\lfloor sc_N\rfloor\leq j\leq \lfloor tc_N\rfloor}B_j}{\frac{c_N}{b_N}}-\mu (t-s)\right)\right)>\epsilon}\\
    \leq \frac{1}{\delta^2}\probability*{\sup_{r\in[t,t+\delta]}\sup_{q\in[s-\delta,s]}\left(\frac{\max_{i\leq N}\sum_{j=\lfloor qc_N\rfloor}^{\lfloor rc_N\rfloor}(A_{i,j}B_j-T_j)}{c_N}-\left(\frac{\max_{\lfloor qc_N\rfloor\leq j\leq \lfloor rc_N\rfloor}B_j}{\frac{c_N}{b_N}}-\mu (r-q)\right)\right)>\epsilon}.
\end{multline*}
Also, 
\begin{multline*}
    \probability*{\sup_{t\in[0,T]}\sup_{s\in[0,t]}\left(-\frac{\max_{i\leq N}\sum_{j=\lfloor sc_N\rfloor}^{\lfloor tc_N\rfloor}(A_{i,j}B_j-T_j)}{c_N}+\left(\frac{\max_{\lfloor sc_N\rfloor\leq j\leq \lfloor tc_N\rfloor}B_j}{\frac{c_N}{b_N}}-\mu (t-s)\right)\right)>\epsilon}\\
    \leq \frac{1}{\delta^2}\probability*{\sup_{r\in[t,t+\delta]}\sup_{q\in[s-\delta,s]}\left(-\frac{\max_{i\leq N}\sum_{j=\lfloor qc_N\rfloor}^{\lfloor rc_N\rfloor}(A_{i,j}B_j-T_j)}{c_N}+\left(\frac{\max_{\lfloor qc_N\rfloor\leq j\leq \lfloor rc_N\rfloor}B_j}{\frac{c_N}{b_N}}-\mu (r-q)\right)\right)>\epsilon}.
\end{multline*}
The proof that these upper bounds converge to 0 as $N\to\infty$ is analogous to the  proof of Lemma \ref{lem: sup max B and process}. 
\end{proof}

\begin{proof}[Proof of Theorem \ref{thm: main result waiting time heavy tail}]
We have proven in Lemma \ref{lem: sup max B and process2} that the maximum waiting time can be approximated with the process $\left(\sup_{s\in[0,t]}\left(\frac{\max_{\lfloor sc_N\rfloor\leq j\leq \lfloor tc_N\rfloor}B_j}{(\frac{c_N}{b_N})}-\mu (t-s)\right),t\in[0,T]\right)$. Therefore, in order to prove convergence of the maximum waiting time to the process $\left(\sup_{s\in[0,t]}(X_{(s,t)}-\mu(t-s)),t\in[0,T]\right)$ in $D[0,T]$, it suffices to prove convergence of the process $\left(\sup_{s\in[0,t]}\left(\frac{\max_{\lfloor sc_N\rfloor\leq j\leq \lfloor tc_N\rfloor}B_j}{(\frac{c_N}{b_N})}-\mu (t-s)\right),t\in[0,T]\right)$ to the process $\left(\sup_{s\in[0,t]}(X_{(s,t)}-\mu(t-s)),t\in[0,T]\right)$ in $D[0,T]$. As in Lemma \ref{lem: conv in D max B}, we again check the conditions given in Lemma \ref{lem: billingsley conv in D}. 

We start with proving the convergence of finite-dimensional distributions. To do this, we show that the joint probabilities of these processes can be written as operations of marginal probabilities, and therefore, the convergence of finite-dimensional distributions follows from the convergence of 1-dimensional distributions. Thus, we can write
\begin{multline}\label{eq: finite dimensional dist}
\probability*{\sup_{s\in[0,t_1]}(X_{(s,t_1)}-\mu(t_1- s))<x_1\cap \sup_{s\in[0,t_2]}(X_{(s,t_2)}-\mu(t_2- s))<x_2}
\\
=\probability*{\sup_{s\in[0,t_1]}(X_{(s,t_1)}+\mu s)<x_1+\mu t_1\Bigg| \sup_{s\in[0,t_2]}(X_{(s,t_2)}+\mu s)<x_2+\mu t_2}\probability*{\sup_{s\in[0,t_2]}(X_{(s,t_2)}+\mu s)<x_2+\mu t_2}.
\end{multline}
Now, we can further rewrite the event $\{\sup_{s\in[0,t_2]}(X_{(s,t_2)}+\mu s)<x_2+\mu t_2\}$ and relate this to the random variable $X_{(s,t_1)}$; namely:

\begin{multline*}
\left\{\sup_{s\in[0,t_2]}(X_{(s,t_2)}+\mu s)<x_2+\mu t_2\right\}\\=\left\{\sup_{s\in[0,t_1]}(X_{(s,t_1)}+\mu s)<x_2+\mu t_2\right\}\cap\Bigg\{X_{(t_1,t_2)}+\mu t_1<x_2+\mu t_2\Bigg\}\cap \left\{\sup_{s\in(t_1,t_2]}(X_{(s,t_2)}+\mu s)<x_2+\mu t_2\right\}.
\end{multline*}
Thus, when  $x_2+\mu t_2\leq x_1+\mu t_1$, then 
$$
\probability*{\sup_{s\in[0,t_1]}(X_{(s,t_1)}+\mu s)<x_1+\mu t_1\Bigg| \sup_{s\in[0,t_2]}(X_{(s,t_2)}+\mu s)<x_2+\mu t_2}=1,
$$
and when $x_2+\mu t_2> x_1+\mu t_1$,
$$
\probability*{\sup_{s\in[0,t_1]}(X_{(s,t_1)}+\mu s)<x_1+\mu t_1\Bigg| \sup_{s\in[0,t_2]}(X_{(s,t_2)}+\mu s)<x_2+\mu t_2}=\frac{\probability*{\sup_{s\in[0,t_1]}(X_{(s,t_1)}+\mu s)<x_1+\mu t_1}}{\probability*{\sup_{s\in[0,t_1]}(X_{(s,t_1)}+\mu s)<x_2+\mu t_2}}.
$$
From now on, we focus on the case $x_2+\mu t_2> x_1+\mu t_1$: the proof of the case $x_2+\mu t_2\leq x_1+\mu t_1$ is analogous. In conclusion,
\begin{multline}\label{eq: limit X s,t}
\probability*{\sup_{s\in[0,t_1]}(X_{(s,t_1)}-\mu(t_1- s))<x_1\cap \sup_{s\in[0,t_2]}(X_{(s,t_2)}-\mu(t_2- s))<x_2}\\
=\frac{\probability*{\sup_{s\in[0,t_1]}(X_{(s,t_1)}-\mu(t_1- s))<x_1}}{\probability*{\sup_{s\in[0,t_1]}(X_{(s,t_1)}-\mu(t_1- s))<x_2+\mu(t_2-t_1)}}\probability*{\sup_{s\in[0,t_2]}(X_{(s,t_2)}-\mu(t_2-s))<x_2}.
\end{multline}
Thus, we can write the joint probability in \eqref{eq: finite dimensional dist} as an operation of marginal probabilities. We can do the same for the process $\left(\sup_{s\in[0,t]}\left(\max_{\lfloor sc_N\rfloor\leq j\leq \lfloor tc_N\rfloor }\frac{B_j}{(\frac{c_N}{b_N})}-\mu(t- s)\right),t\in[0,T]\right)$:
\begin{align}
&\probability*{\sup_{s\in[0,t_1]}\left(\max_{\lfloor sc_N\rfloor\leq j\leq \lfloor t_1c_N\rfloor }\frac{B_j}{\frac{c_N}{b_N}}-\mu(t_1- s)\right)<x_1\cap\sup_{s\in[0,t_2]}\left(\max_{\lfloor sc_N\rfloor\leq j\leq \lfloor t_2c_N\rfloor }\frac{B_j}{\frac{c_N}{b_N}}-\mu(t_2- s)\right)<x_2}\label{subeq: limit B s,t}\\
& \quad=\frac{\probability*{\sup_{s\in[0,t_1]}(\max_{\lfloor sc_N\rfloor\leq j\leq \lfloor t_1c_N\rfloor }\frac{B_j}{(\frac{c_N}{b_N})}-\mu(t_1- s))<x_1}}{\probability*{\sup_{s\in[0,t_1]}(\max_{\lfloor sc_N\rfloor\leq j\leq \lfloor t_1c_N\rfloor }\frac{B_j}{(\frac{c_N}{b_N})}-\mu(t_1- s))<x_2+\mu(t_2-t_1)}}\nonumber\\
& \quad\quad\quad\quad\quad\quad\quad\quad\quad\quad\quad\quad\quad\quad\quad\quad\quad\quad\quad\quad\quad\quad\quad\quad\quad\cdot\probability*{\sup_{s\in[0,t_2]}\left(\max_{\lfloor sc_N\rfloor\leq j\leq \lfloor t_2c_N\rfloor }\frac{B_j}{\frac{c_N}{b_N}}-\mu(t_2-s)\right)<x_2}.
\end{align}
Using Lemma \ref{lem: conv in D max B} and the decomposition of a joint probability into marginal probabilities, we establish that the probability in \eqref{subeq: limit B s,t} converges to the probability in \eqref{eq: limit X s,t} as $N\to\infty$. Analogous extensions hold for higher dimensional distributions. Hence; convergence of finite-dimensional distributions follows. To prove process convergence, we show that the second and third condition of Lemma \ref{lem: billingsley conv in D} also hold.
To establish that the second condition holds, we observe that the following bound holds:
\begin{multline*}
    \probability*{\bigg|\sup_{s\in[0,T]}(X_{(s,T)}-\mu(T-s))-\sup_{s\in[0,T-\delta]}(X_{(s,T-\delta)}-\mu(T-\delta-s))\bigg|>\epsilon}\\
    \leq \probability*{\sup_{s\in[0,T]}(X_{(s,T)}+\mu s)-\sup_{s\in[0,T-\delta]}(X_{(s,T-\delta)}+\mu s)+\mu\delta>\epsilon}.
\end{multline*}
Now, we can further simplify this as follows:
\begin{align*}
    &\sup_{s\in[0,T]}(X_{(s,T)}+\mu s)-\sup_{s\in[0,T-\delta]}(X_{(s,T-\delta)}+\mu s)+\mu \delta\\
    &\quad=\max\left(\sup_{s\in[0,T-\delta]}(X_{(s,T)}+\mu s),\sup_{s\in[T-\delta,T]}(X_{(s,T)}+\mu s)\right)-\sup_{s\in[0,T-\delta]}(X_{(s,T-\delta)}+\mu s) +\mu \delta\\
    &\quad\leq \max\left(\sup_{s\in[0,T-\delta]}(X_{(s,T)}-X_{(s,T-\delta)}),X_{(T-\delta,T)}+\mu T-\sup_{s\in[0,T-\delta]}(X_{(s,T-\delta)}+\mu s) \right)+\mu\delta\\
    &\quad\leq \max\left(\sup_{s\in[0,T-\delta]}(X_{(s,T)}-X_{(s,T-\delta)}),X_{(T-\delta,T)}+\mu T-\mu(T-\delta) \right)+\mu\delta\\
    &\quad=X_{(T-\delta,T)}+2\mu\delta.
\end{align*}
We have that 
$$
\probability{X_{(T-\delta,T)}+2\mu\delta>\epsilon}=1-\exp\left(-\frac{\delta}{(\epsilon-2\mu\delta)^{\beta}}\right)\overset{\delta\downarrow 0}{\longrightarrow} 0.
$$
To establish that the third condition holds, we first observe that for $r\leq t$
\begin{multline*}
\bigg|\sup_{s\in[0,t]}\left(\frac{\max_{\lfloor sc_N \rfloor\leq j\leq \lfloor tc_N \rfloor}B_j}{\frac{c_N}{b_N}}-\mu(t-s)\right)-\sup_{s\in[0,r]}\left(\frac{\max_{\lfloor sc_N \rfloor\leq j\leq \lfloor rc_N \rfloor}B_j}{\frac{c_N}{b_N}}-\mu(r-s)\right)\bigg|\\
\leq \sup_{s\in[0,t]}\left(\frac{\max_{\lfloor sc_N \rfloor\leq j\leq \lfloor tc_N \rfloor}B_j}{\frac{c_N}{b_N}}+\mu s\right)-\sup_{s\in[0,r]}\left(\frac{\max_{\lfloor sc_N \rfloor\leq j\leq \lfloor rc_N \rfloor}B_j}{\frac{c_N}{b_N}}+\mu s\right)+\mu(t-r).
\end{multline*}
With an analogous derivation as in the proof of Lemma \ref{lem: conv in D max B}, we see that the third condition holds. Thus, we have process convergence.

\end{proof}
\section{Steady-state convergence of the maximum waiting time}\label{sec: steady state}
Finally, we prove steady-state convergence of the longest of the $N$ waiting times. We give lower and upper bounds of $\probability{\max_{i\leq N}W_i(\infty)>xc_N}$ and show that these are asymptotically tight.
\begin{proof}[Proof of Theorem \ref{thm: steady state heavy tail}]
First of all, to prove a sharp lower bound, we first notice that $$
\max_{i\leq N}W_i(\infty)=\max_{i\leq N}\sup_{k\geq 0}\sum_{j=1}^k(A_{i,j}B_j-T_j).$$ Thus, the maximum steady-state waiting time is lower bounded by random variables of the form\\ $\max_{i\leq N}\sup_{0\leq k\leq l}\sum_{j=1}^k(A_{i,j}B_j-T_j)$ with $l>0$. Thus, by using the convergence result in \eqref{eq: time convergence auxiliary process} in Theorem \ref{thm: auxiliary result waiting time heavy tail}, we know that 
$$
\liminf_{N\to\infty}\probability*{\max_{i\leq N}W_i(\infty)>xc_N}\geq \lim_{N\to\infty}\probability*{\max_{i\leq N}\sup_{t\in[0,M]}\sum_{j=1}^{\lfloor tc_N\rfloor}(A_{i,j}B_j-T_j)>xc_N}= \probability*{\sup_{t\in[0,M]}(X_t-\mu t)>x}.
$$
Now, following Equations \eqref{eq: steady state explicit prob} and \eqref{eq: transient explicit prob} in Proposition \ref{prop: marginal steady state}, it is easy to see that
$$
\probability*{\sup_{t\in[0,M]}(X_t-\mu t)>x}\to\probability*{\sup_{t>0}(X_t-\mu t)>x},
$$
as $M\to\infty$. Thus, we have a tight lower bound.

Second, we want to find a tight upper bound for the tail probability of the steady-state maximum queue length. We have that
\begin{align}
&\probability*{\max_{i\leq N}W_i(\infty)>xc_N}\nonumber\\
&\quad=\probability*{\max_{i\leq N}\sup_{k\geq 0}\sum_{j=1}^{k}(A_{i,j}B_j-T_j)>xc_N}\nonumber\\
&\quad=\probability*{\max\left(\max_{i\leq N}\sup_{t\in[0,M]}\sum_{j=1}^{\lfloor tc_N\rfloor}(A_{i,j}B_j-T_j),\max_{i\leq N}\sup_{t>M}\sum_{j=1}^{\lfloor tc_N\rfloor}(A_{i,j}B_j-T_j)\right)>xc_N}\nonumber\\
&\quad\leq \probability*{\max_{i\leq N}\sup_{t\in[0,M]}\sum_{j=1}^{\lfloor tc_N\rfloor}(A_{i,j}B_j-T_j)>xc_N}+\probability*{\max_{i\leq N}\sup_{t>M}\sum_{j=1}^{\lfloor tc_N\rfloor}(A_{i,j}B_j-T_j)>xc_N}.\label{subeq: upper bound steady state}
\end{align}
For the first term in \eqref{subeq: upper bound steady state} we obtain that
$$
\probability*{\max_{i\leq N}\sup_{t\in[0,M]}\sum_{j=1}^{\lfloor tc_N\rfloor}(A_{i,j}B_j-T_j)>xc_N}\LimitN \probability*{\sup_{t\in[0,M]}X_t-\mu t>x}\to 1-\exp\left(-\frac{1}{\mu(\beta-1)x^{\beta-1}}\right),
$$
as $M\to\infty$. Thus, we need to prove that the second term in \eqref{subeq: upper bound steady state} asymptotically vanishes when $N,M\to\infty$.
Let $\hat{A}_{i,j}$ and $\hat{B}_j$ be independent copies of $A_{i,j}$ and $B_j$ respectively. Then, we can bound the second term in \eqref{subeq: upper bound steady state} as follows:
\begin{align} 
    &\probability*{\max_{i\leq N}\sup_{t>M}\sum_{j=1}^{\lfloor tc_N\rfloor}(A_{i,j}B_j-T_j)>xc_N}\nonumber\\
    &\quad=\probability*{\max_{i\leq N}\left(\sum_{j=1}^{\lfloor Mc_N\rfloor}(A_{i,j}B_j-T_j)+\sup_{k\geq 0}\sum_{j=1}^k(\hat{A}_{i,j}\hat{B}_j-\hat{T}_j)\right)>xc_N}\nonumber\\
    &\quad\leq \probability*{\max_{i\leq N}\sum_{j=1}^{\lfloor Mc_N\rfloor}(A_{i,j}B_j-T_j)+\max_{i\leq N}\sup_{k\geq 0}\sum_{j=1}^k(\hat{A}_{i,j}\hat{B}_j-\hat{T}_j)>xc_N}\label{subeq: upper bound 1}\\
    &\quad \leq \probability*{\max_{i\leq N}\sum_{j=1}^{\lfloor Mc_N\rfloor}(A_{i,j}B_j-T_j)>-\frac{\mu}{2}Mc_N}+\probability*{\max_{i\leq N}\sup_{k\geq 0}\sum_{j=1}^k(\hat{A}_{i,j}\hat{B}_j-\hat{T}_j)>\left(x+\frac{\mu}{2}M\right)c_N}.\label{subeq: upper bound 2}
\end{align}
The bound in \eqref{subeq: upper bound 1} holds because $\max_{i\leq N}(a_i+b_i)\leq \max_{i\leq N}a_i+\max_{i\leq N}b_i$. The bound in \eqref{subeq: upper bound 2} follows from the union bound.
For the first term in \eqref{subeq: upper bound 2} we have that
$$
\probability*{\max_{i\leq N}\sum_{j=1}^{\lfloor Mc_N\rfloor}(A_{i,j}B_j-T_j)>-\frac{\mu}{2}Mc_N}\LimitN 1-\exp\left(\frac{-M}{(\frac{\mu M}{2})^{\beta}}\right)\longrightarrow 0,
$$
as $M\to\infty$. In order to analyze the second term in \eqref{subeq: upper bound 2}, we use the fact that $\mathbb{E}[A_{i,j}B_j-T_j]=-\mu<0$, from this, it follows that there exists a $\gamma>1$, such that $\mathbb{E}[A_{i,j}B_j]<\mathbb{E}[T_j]/\gamma$. We write $\gamma\mathbb{E}[A_{i,j}B_j]-\mathbb{E}[T_j]=-\mu_{\gamma}<0$. Then,
\begin{align}
    &\probability*{\max_{i\leq N}\sup_{k\geq 0}\sum_{j=1}^k(\hat{A}_{i,j}\hat{B}_j-\hat{T}_j)>\left(x+\frac{\mu}{2}M\right)c_N}\nonumber\\
    &\quad\leq \probability*{\max_{i\leq N}\sup_{k\in[0,\lfloor c_N\rfloor]}\sum_{j=1}^k(\hat{A}_{i,j}\hat{B}_j-\hat{T}_j)>\left(x+\frac{\mu}{2}M\right)c_N}\nonumber\\
        &\quad\quad+\sum_{n=0}^{\infty}\probability*{\max_{i\leq N}\sup_{k\in[\lfloor \gamma^nc_N\rfloor,\lfloor \gamma^{n+1}c_N\rfloor]}\sum_{j=1}^k(\hat{A}_{i,j}\hat{B}_j-\hat{T}_j)>\left(x+\frac{\mu}{2}M\right)c_N}\nonumber\\
      &\quad  \leq  \probability*{\max_{i\leq N}\sup_{k\in[0,\lfloor c_N\rfloor]}\sum_{j=1}^k(\hat{A}_{i,j}\hat{B}_j-\hat{T}_j)>\left(x+\frac{\mu}{2}M\right)c_N}\nonumber\\
      &\quad\quad+ \sum_{n=0}^{\infty}\probability*{\max_{i\leq N}\sum_{j=1}^{\lfloor \gamma^{n+1}c_N\rfloor}\hat{A}_{i,j}\hat{B}_j-\sum_{j=1}^{\lfloor \gamma^{n}c_N\rfloor}\hat{T}_j>\left(x+\frac{\mu}{2}M\right)c_N}\nonumber\\
    &\quad \LimitN  \probability*{\sup_{t\in[0,1]}(X_t-\mu t)>x+\frac{\mu}{2}M}+\sum_{n=0}^{\infty}\left(1-\exp\left(-\frac{\gamma^{n+1}}{(x+\frac{\mu M}{2}+\gamma^n\mu_{\gamma})^{\beta}}\right)\right).\label{subeq: limit upper bound steady state}
\end{align}
It is clear that $\probability{\sup_{t\in[0,1]}(X_t-\mu t)>x+\frac{\mu M}{2}}\longrightarrow 0$ as $M\to\infty$. The sum in \eqref{subeq: limit upper bound steady state} is finite and also converges to 0 as $M\to\infty$, as the ratio test gives us that 
$$
\lim_{n\to\infty}\frac{\left(1-\exp\left(-\frac{\gamma^{n+2}}{(x+\frac{\mu M}{2}+\gamma^{n+1}\mu_{\gamma})^{\beta}}\right)\right)}{\left(1-\exp\left(-\frac{\gamma^{n+1}}{(x+\frac{\mu M}{2}+\gamma^n\mu_{\gamma})^{\beta}}\right)\right)}= \frac{1}{\gamma^{\beta-1}}<1.
$$
Hence, we can choose for all $\epsilon>0$ a $K$ large enough such that
$$
\sum_{n=K}^{\infty}\left(1-\exp\left(-\frac{\gamma^{n+1}}{(x+\frac{\mu M}{2}+\gamma^n\mu_{\gamma})^{\beta}}\right)\right)<\epsilon,
$$
and it is obvious that
$$
\sum_{n=0}^K\left(1-\exp\left(-\frac{\gamma^{n+1}}{(x+\frac{\mu M}{2}+\gamma^n\mu_{\gamma})^{\beta}}\right)\right)\longrightarrow 0,
$$
as $M\to\infty$.
Thus, we can conclude that the both terms in \eqref{subeq: limit upper bound steady state} converge to 0 as $M\to\infty$, and consequently, both terms in \eqref{subeq: upper bound 2} asymptotically vanish. Returning to the upper bound for the steady-state tail probability of the maximum waiting time given in \eqref{subeq: upper bound steady state}, we can conclude that
$$
\limsup_{N\to\infty}\probability*{\max_{i\leq N}W_i(\infty)>xc_N}\leq \probability*{\sup_{t>0}(X_t-\mu t)>x}.
$$
\end{proof}
We have proven process convergence of the maximum transient waiting time and we have proven steady state convergence. The limiting processes have the form of a supremum of Fr\'echet-distributed random variables with a negative drift. We now give an explicit expression of the cumulative distribution function.
\begin{proof}[Proof of Proposition \ref{prop: marginal steady state}]
To prove Equation \eqref{eq: steady state explicit prob}, we provide sharp lower and upper bounds of $\probability{\sup_{t>0}(X_t-\mu t)<x}$. First of all, let $\delta>0$. We have that
$$
\probability*{X_{\delta}-\mu \delta<x}=\exp\left(-\frac{\delta}{(x+\mu\delta)^{\beta}}\right).
$$
Obviously, we can bound $\probability*{\sup_{t>0}(X_t-\mu t)<x}$ from above as
$$
\probability*{\sup_{t>0}(X_t-\mu t)<x}<\probability*{\cap_{i=1}^{\infty}X_{i\delta}-\mu i\delta<x}.
$$
We can write $X_{2\delta}=\max(\hat{X}_{\delta},X_{\delta})$. From this relation, we know that, if $X_{\delta}-\delta<x$, then $X_{2\delta}-2\delta<x$ if and only if $\hat{X}_{\delta}-2\delta<x$. Therefore, 
$$
\probability{X_{\delta}-\delta<x\cap X_{2\delta}-2\delta<x }=\probability{X_{\delta}-\delta<x\cap \hat{X}_{\delta}-2\delta<x }=\probability{X_{\delta}-\delta<x}\probability{\hat{X}_{\delta}-2\delta<x }.
$$
Thus, in general, the cumulative distribution function of $\sup_{t>0}(X_t-\mu t)$ is bounded from above as
\begin{align}
    \probability*{\sup_{t>0}(X_t-\mu t)<x}<\probability*{\cap_{i=1}^{\infty}X_{i\delta}-\mu i\delta<x}=\prod_{i=1}^{\infty}\exp\left(-\frac{\delta}{(x+\mu i\delta)^{\beta}}\right).\label{eq: upper bound exact expression steady state}
\end{align}
We can find a lower bound as well, because both $X_t$ and $\mu t$ are non-decreasing in $t$ we know that $\sup_{s\in((i-1)\delta,i\delta]}(X_{s}-\mu s)\leq X_{i\delta}-\mu(i-1)\delta $. Therefore,
$$
\probability*{\sup_{t>0}(X_t-\mu t)<x}=\probability*{\cap_{i=1}^{\infty}\sup_{s\in((i-1)\delta,i\delta]}(X_{s}-\mu s)<x}>\probability*{\cap_{i=1}^{\infty}X_{i\delta}-\mu (i-1)\delta<x}.
$$
With a similar derivation as before, we have that
$$
\probability*{\cap_{i=1}^{\infty}X_{i\delta}-\mu (i-1)\delta<x}=\prod_{i=1}^{\infty}\exp\left(-\frac{\delta}{(x+\mu (i-1)\delta)^{\beta}}\right).
$$
Now, we can rewrite this expression as
$$
\prod_{i=0}^{\infty}\exp\left(-\frac{\delta}{(x+\mu i\delta)^{\beta}}\right)=\exp\left(-\frac{\delta}{(\mu\delta)^{\beta}}\sum_{i=0}^{\infty}\frac{1}{\left(\frac{x}{(\mu\delta)}+i\right)^{\beta}}\right)=\exp\left(-\frac{\delta}{(\mu\delta)^{\beta}}\zeta\left(\beta,\frac{x}{\mu\delta}\right)\right),
$$
where $\zeta(\beta,x)$ is the Hurwitz zeta function, cf.\ \cite[Eq. (1.10)]{adamchik1998some}. We have that
$$
\lim_{\delta \downarrow 0}\frac{\delta}{(\mu\delta)^{\beta}}\zeta\left(\beta,\frac{x}{\mu\delta}\right)=\frac{1}{\mu(\beta-1)x^{\beta-1}}.
$$
The same limit holds for the upper bound in \eqref{eq: upper bound exact expression steady state}, thus Equation \eqref{eq: steady state explicit prob} follows. The proof of Equation \eqref{eq: transient explicit prob} is analogous, and follows from the fact that
$$
\lim_{\delta \downarrow 0}\frac{\delta}{(\mu\delta)^{\beta}}\sum_{i=0}^{\lfloor\frac{t}{\delta}\rfloor}\frac{1}{\left(\frac{x}{(\mu\delta)}+i\right)^{\beta}}=\frac{1}{\mu^{\beta}(\beta-1)}\left(\frac{1}{(x/\mu)^{\beta-1}}-\frac{1}{(x/\mu+t)^{\beta-1}}\right).
$$
\end{proof}
\subsection*{Acknowledgments}
This work is part of the research program Complexity in high-tech manufacturing, (partly) financed by the Dutch Research Council (NWO) through contract 438.16.121.
\newpage
\pagestyle{plain}
\bibliographystyle{plain}

\bibliography{refs}
\end{document}